\documentclass[leqno]{amsproc}
\usepackage{amsfonts,amssymb,amsmath,amsgen,amsthm}
\usepackage{hyperref}
\usepackage{color}
\usepackage{epsfig,epstopdf,color,bm}
\usepackage{graphicx,wrapfig}
\newcommand{\msc}[2][2000]{%
  \let\@oldtitle\@title%
  \gdef\@title{\@oldtitle\footnotetext{#1 \emph{Mathematics subject
        classification.} #2}}%
}

\def\dis
{\displaystyle}

\def\C{{\mathbb C}}
\def\R{{\mathbb R}}
\def\T{{\mathbb T}}
\def\H{{\mathcal H}}
\def\Sch{{\mathcal S}}
\def\O{\mathcal O}
\def\F{\mathcal F}

\def\({\left(}
\def\){\right)}
\def\<{\left\langle}
\def\>{\right\rangle}
\def\le{\leqslant}
\def\ge{\geqslant}

\def\Eq#1#2{\mathop{\sim}\limits_{#1\rightarrow#2}}
\def\Tend#1#2{\mathop{\longrightarrow}\limits_{#1\rightarrow#2}}

\def\d{{\partial}}
\def\eps{\varepsilon}

\def\si{{\sigma}}

\DeclareMathOperator{\RE}{Re}
\DeclareMathOperator{\IM}{Im}

\newcommand{\DD}{\mathcal D}
\newcommand{\EE}{\mathcal E}
\newcommand{\EBD}{{\mathcal E}_{\rm BD}}
\newcommand{\DBD}{{\mathcal D}_{\rm BD}}

\newtheorem{theorem}{Theorem}[section]
\newtheorem{corollary}[theorem]{Corollary}
\newtheorem{lemma}[theorem]{Lemma}

\theoremstyle{definition}

\newtheorem{example}[theorem]{Example}
\newtheorem{remark}[theorem]{Remark}

\numberwithin{equation}{section}

\author[R. Carles]{R\'emi Carles}

\address{Univ Rennes, CNRS\\ IRMAR - UMR 6625\\ F-35000
  Rennes, France}
\email{Remi.Carles@math.cnrs.fr}

\title[LogNLS and isothermal fluids]{Logarithmic Schr\"odinger
  equation and isothermal fluids}

\begin{document}

\maketitle

\begin{abstract}
  We consider the large time behavior in two types of
equations, posed on the whole space $\R^d$: the Schr\"odinger equation
with a logarithmic nonlinearity on the one hand; compressible,
isothermal, Euler, Korteweg and
quantum Navier-Stokes equations on the other hand. We explain some
connections between the two families of equations, and show how these
connections may help having an insight in all cases. We insist on some
specific aspects only, and refer to the cited articles for more
details, and more complete statements. We try to give a general
picture of the results, and present 
some heuristical arguments that can help the intuition,  which are not
necessarily found in the mentioned articles. 
\end{abstract}

\section{Introduction}
\label{sec:intro}

\subsection{Linear equations}
As a preliminary, and for future comparison with the logarithmic
Schr\"odinger equation, we recall some basic facts regarding the large
time dynamics for the linear heat equation and the linear
Schr\"odinger equation, on $\R^d$.
\subsubsection{Heat equation}
Consider the Cauchy problem
\begin{equation*}
 \d_t u = \frac{1}{2}\Delta u ,\quad x\in
      \R^d,\quad u_{\mid t=0}=u_0\in L^1(\R^d). 
  \end{equation*}
The solution is given by the explicit formula: $\dis u(t,x) = \frac{1}{(2\pi
    t)^{d/2}}\int_{\R^d}e^{-\frac{|x-y|^2}{2t}}u_0(y)dy.$ This formula
is classical, and  follows from Fourier analysis, see
e.g. \cite{Rauch91}. To fix ideas, we normalize the Fourier transform
as
\begin{equation*}
  \hat f(\xi) = \frac{1}{(2\pi)^{d/2}}\int_{\R^d} e^{-ix\cdot
    \xi}f(x)dx,\quad f\in \Sch(\R^d). 
\end{equation*}
The Fourier transform of $u$ at time $t$ is given by
\begin{equation*}
  \hat u(t,\xi)= e^{-\frac{t}{2}|\xi|^2}\hat u_0(\xi) =
  \underbrace{e^{-\frac{t}{2}|\xi|^2}\hat u_0(0)}_{\text{order }t^{-d/(2p)}
    \text{ in }L^p} +
  \underbrace{e^{-\frac{t}{2}|\xi|^2}\(\hat u_0(\xi)-\hat
    u_0(0)\)}_{\O\(|\xi| e^{-\frac{t}{2}|\xi|^2}\): \text{ order }t^{-(d+1)/(2p)}
    \text{ in }L^p} .
\end{equation*}
If $m:=\int_{\R^d}u_0=(2\pi)^{d/2}\hat u_0(0)\not =0$, we infer, by the Fourier inverse formula,
\begin{equation*}
u (t,x) \Eq t \infty \frac{m}{(2\pi t)^{d/2}}
      e^{-|x|^2/(2t)},
\end{equation*}
and we leave out the discussion on the norms involved above: the key
message is that the large time description involves a universal
diffusive rate, and a universal
Gaussian profile, the initial data appears only through its mass $m$. 

\subsubsection{Schr\"odinger equation}\label{sec:LS}
For the Schr\"odinger equation, the
initial datum naturally belongs to $L^2$,
  \begin{equation*}
i\d_t u + \frac{1}{2}\Delta u =0,\quad x\in
      \R^d,\quad u_{\mid t=0}=u_0\in L^2(\R^d). 
  \end{equation*}
Again, the solution is given explicitly, now by an oscillatory
integral:
\[u(t,x) = \frac{1}{(2i\pi
    t)^{d/2}}\int_{\R^d}e^{i\frac{|x-y|^2}{2t}}u_0(y)dy.\]
 We emphasize two consequences:
\begin{itemize}
\item Dispersion: $\dis 
    \|u(t)\|_{L^\infty(\R^d)}\lesssim
    \frac{1}{|t|^{d/2}}\|u_0\|_{L^1(\R^d)}.$ 
\item Large time description: $\left\lVert u(t) -
    A(t)u_0\right\rVert_{L^2(\R^d)} \Tend t 
    {\pm\infty} 0$, where
  \begin{equation*}
 A(t)u_0(x) = \frac{1}{(it)^{d/2}}\hat
  u_0\(\frac{x}{t}\)e^{i\frac{\lvert x\rvert^2}{2t}}.
  \end{equation*}
We now have a universal oscillation, but the
profile depends on the initial 
data, through its Fourier transform.
\end{itemize}
To check the second point, expanding the argument of the exponential
in the oscillatory integral, we can write
\begin{equation*}
  u(t,x) = M_t D_t \F M_t u_0(x),
\end{equation*}
where $\F$ stands for the Fourier transform, $M_t$ is the
multiplication by $e^{i|x|^2/(2t)}$, and $D_t$ is the time dependent
dilation
\begin{equation*}
  D_t\varphi(x) = \frac{1}{(it)^{d/2}}\varphi\(\frac{x}{t}\).
\end{equation*}
Then $A(t)u_0 =  M_t D_t \F u_0$, and the approximation follows from
the limit $M_t \to 1$ in $L^2$, as $t\to \infty$, as granted by the
dominated convergence theorem. 
\begin{example}[Explicit computation in the Gaussian
  case]\label{ex:gaussian}
  The evolution of Gaussian initial data is given, for $z\in \C$ such that
  $ \RE z>0$, by:
\begin{equation*}
 e^{i\frac{t}{2}\Delta}\(e^{-z\frac{|x|^2}{2}}\) = \frac{1}{(1+itz)^{d/2}}
  e^{-\frac{z}{1+itz}\frac{|x|^2}{2}}.  
\end{equation*}
\end{example}

\subsection{Nonlinear Schr\"odinger equation: the usual nonlinearity}\label{sec:NLS}

We recall a few standard properties related to the nonlinear
Schr\"odinger equation with power-like nonlinearity. These properties
can be found e.g. in \cite{CazCourant}. 
For $\lambda\in \R$, and
$0<\si<\frac{2}{(d-2)_+}$, consider: 
   \begin{equation}\label{eq:NLS}
i\d_t u + \frac{1}{2}\Delta u =\lambda |u|^{2\si}u ,\quad x\in
      \R^d,\quad u_{\mid t=0}=u_0\in H^1(\R^d). 
    \end{equation}
 Under our assumption on $\si$, the nonlinearity is $H^1$-subcritical
 (the $L^{2\si+2}$-norm in the energy below is controlled by the
 $H^1$-norm, thanks to Sobolev embedding),
 and the Cauchy problem \eqref{eq:NLS} is locally well-posed in
 $H^1(\R^d)$.   The standard proof relies on Strichartz estimates,
 and a fixed point argument on Duhamel's formula; the nonlinearity is
 thus considered as a perturbation,  the problem is
 semilinear. The situation is completely different in the case of
 equations from (compressible) fluid mechanics, addressed in the second part of this survey.
  \subsubsection{Invariants}
  Equation~\ref{eq:NLS} is invariant under space and time translation,
  as well as under the gauge transforms
\[u(t,x)\mapsto e^{i\theta}u(t,x),\quad \theta\in \R.\]
  The  Galilean invariance reads as follows:  if $u(t,x)$ solves \eqref{eq:NLS}, then  for
  any ${\mathbf v}\in \R^d$, so does
  \begin{equation}
    \label{eq:galilee}
    u(t,x-{\mathbf v} t)e^{i{\mathbf v}\cdot x-i|{\mathbf v}|^2 t/2}.
  \end{equation}
  (With a different initial datum.) This property is useful to
  construct multisolitons (with different velocities), as in e.g. \cite{MartelMerle2006}.\\
 The following quantities are formally independent of time:
 \begin{align*}
 \text{Mass: } & M(u(t)):=\|u(t)\|_{L^2(\R^d)}^2,\\
  \text{Momentum: }& J(u(t)):= \IM\int_{\R^d}\bar u(t,x)\nabla u(t,x)dx,\\
\text{Energy: }  & E(u(t)):=\frac{1}{2}\|\nabla
    u(t)\|_{L^2(\R^d)}^2+\frac{\lambda}{\si+1}\|u(t)\|_{L^{2\si+2}(\R^d)}^{2\si+2}
    .
\end{align*} 
 According to the sign of $\lambda$, the energy is or is not a
 positive functional.  
\subsubsection{Defocusing case}
 If $\lambda>0$: the local well-posedness in $H^1$ and the
 conservations of mass and energy imply global existence ($u\in
    L^\infty(\R;H^1(\R^d))$), and if $\si>2/d$, the solution behaves
    asymptotically like a linear solution,
  \begin{equation}\label{eq:scattering}
 \exists u_+\in H^1(\R^d),\quad \|u(t)-e^{i\frac{t}{2}\Delta}
      u_+\|_{H^1(\R^d)}\Tend t \infty 0.
  \end{equation}
The (inverse of) the wave operator is not trivial: $u_0\mapsto u_+$
is one-to-one. In view of the description in the linear case, this
means that the large time asymptotics involves a universal oscillation, a
universal dispersion, and a somehow arbitrary profile. 
\smallbreak

If we assume in addition that $x\mapsto |x|u_0(x)$ belongs to
$L^2(\R^d)$, then the same conclusion as above is known to remain
valid for
some smaller values of $\si$, but \emph{not too small}. Typically, for
$\si\le 1/d$, if $u$ solves \eqref{eq:NLS}, then its large time
behavior cannot be directly compared to the linear evolution, in the
sense that if there exists $u_+\in L^2(\R^d)$ such that
\[ \|u(t)-e^{i\frac{t}{2}\Delta}
  u_+\|_{L^2(\R^d)}\Tend t \infty 0,\]
then necessarily $u_0=u_+=0$ (hence $u\equiv 0$), from
\cite{Barab}. This is due to the presence of \emph{long range
  effects}, and scattering theory must be modified, see
e.g. \cite{HN06} and references therein.  We will give reasons to
consider that the limit $\si\to 0$ leads to the logarithmic
Schr\"odinger equation (see Section~\ref{sec:limit}), and show that
the dynamical properties related 
to that model are very specific.
\subsubsection{Focusing case}

If $\lambda<0$,  finite time blow-up is possible when $\si\ge2/d$, as
proved typically by a virial computation (the second order derivative
of the function  $t\mapsto\|x
u(t)\|_{L^2(\R^d)}^2$ may be smaller than a negative constant,
\cite{Glassey}, see also 
\cite{CazCourant}), and blow-up is characterized 
by the existence of a finite $T^*$ such that
  \begin{equation*}
\lim_{t\to T^*}\|\nabla u(t)\|_{L^2}=\infty.
\end{equation*}
For $\si\ge 2/d$, small initial data generate global solutions, which
are moreover asymptotically linear in the sense of
\eqref{eq:scattering}.

Finally, we evoke the existence of large standing waves, of the form $u(t,x) = e^{i\omega
  t}\psi(x)$. When $\psi$ is a ground state (which is unique up to the
invariants of the associated elliptic equation), the above standing
wave is orbitally stable if and only if  $\si<2/d$ (instability by
blow-up occurs when $\si\ge 2/d$, \cite{BeCa81,Weinstein83}: small --
in $H^1$ --
perturbations of the standing wave may generate a solution which blows
up in finite time). The
right notion is indeed orbital stability, as opposed to asymptotic
stability of the standing wave, due to the invariants of the equation:
in view of the Galilean invariance \eqref{eq:galilee}, a small initial
perturbation $\psi(x)e^{i\mathbf v\cdot x}$ ($|\mathbf v|\ll1$) will
generate a standing wave whose ``support'' becomes distinct from the
support of $u$ for sufficiently large time. Orbital stability consists
in taking the invariants of the equation into account: in the present
case, this means that for any $\eps>0$, there exists $\eta>0$ such
that if
\begin{equation*}
  \text{if}\quad \|\tilde u_0-\psi\|_{H^1}\le \eta,
\quad\text{then}\quad  \sup_{t\in \R}\inf_{\theta\in
    \R}\inf_{y\in \R^d}\|\tilde u(t) - 
  e^{i\theta} \phi(\cdot -y)\|_{H^1}\le \eps,
\end{equation*}
where $\tilde u$ is the solution to \eqref{eq:NLS} with initial data
$\tilde u_0$. See e.g. \cite{CazCourant}.
\smallbreak

We will see that there are many differences in the case where the
power-like nonlinearity is replaced by a logarithmic nonlinearity. 

\subsection{Logarithmic Schr\"odinger equation}\label{sec:logNLS}
We now consider the Cauchy problem
\begin{equation}
  \label{eq:logNLS}
  i\d_t u +\frac{1}{2} \Delta u =\lambda \ln\(|u|^2\)u  ,\quad u_{\mid t=0} =u_0  ,
\end{equation}
with $x\in \R^d$, $d\ge 1$, and $\lambda\in \R$.
This model was introduced in \cite{BiMy76} to
satisfy the following tensorization property: if the initial datum is a tensor
product,
\begin{equation*}
  u_0(x) =\prod_{j=1}^d u_{0j}(x_j),
\end{equation*}
then the solution to \eqref{eq:logNLS} is given by
\begin{equation*}
  u(t,x) =\prod_{j=1}^d u_{j}(t,x_j), 
\end{equation*}
where each $u_j$ solves a one-dimensional equation,
\begin{equation*}
   i\d_t u_j +\frac{1}{2} \d_{x_j}^2 u_j = \lambda
   \ln\(|u_j|^2\)u_j  ,\quad u_{j\mid t=0} =u_{0j} . 
 \end{equation*}
The logarithmic nonlinearity turns out to be the only one satisfying such a
property. This nonlinearity has then been proposed to model various
physical phenomena, e.g.  quantum optics \cite{buljan, KEB00}, nuclear physics \cite{Hef85}, transport and diffusion phenomena \cite{DMFGL03, hansson},
open quantum systems \cite{yasue, HeRe80}, effective quantum gravity
\cite{Zlo10}, theory of superfluidity and 
Bose-Einstein condensation \cite{BEC}. This tensorization property is
classical when the linear Schr\"odinger equation is considered, and
might suggest that nonlinear effects in \eqref{eq:logNLS} are weak: we
will see that on the contrary, the dynamical properties associated to
\eqref{eq:logNLS} are rather unique. This is due to the singularity
of the logarithm at the origin. 
\smallbreak

Like above, \eqref{eq:logNLS} is invariant under space and time
translations, gauge and Galilean transforms. We still have
conservation of mass, momentum and energy, but the expression
of the latter has changed:
\begin{align*}
& M(u(t))=\|u(t)\|_{L^2(\R^d)}^2,\\
 & J(u(t))= \IM\int_{\R^d}\bar u(t,x)\nabla u(t,x)dx,\\
& E(u(t))=\frac{1}{2}\|\nabla u(t)\|_{L^2(\R^d)}^2+\lambda\int_{\R^d}
    |u(t,x)|^2\(\ln|u(t,x)|^2-1\)dx.
\end{align*}
In view of the conservation of mass, we rather consider the energy
\begin{equation}
  \label{eq:ElogNLS}
   E(u(t)):=\frac{1}{2}\|\nabla u(t)\|_{L^2(\R^d)}^2+\lambda\int_{\R^d}
    |u(t,x)|^2\ln|u(t,x)|^2dx.
\end{equation}
The energy has no definite sign, and it is not completely clear to
decide about the influence of the sign of $\lambda$ on  the
dynamics. However, a formal argument suggests that when $\lambda<0$,
solutions cannot disperse: indeed, if $u$ is dispersive, then it
morally goes to zero pointwise, and the argument of the logarithm in
the energy going to zero, the factor of $\lambda$ goes to $-\infty$:
if $\lambda<0$, this contradicts the conservation of the energy, which
would then become infinite. On the other hand, if $\lambda>0$, this
argument shows that if a solution is dispersive, then its $\dot
H^1$-norm becomes infinite in the large time r\'egime. 
\smallbreak

Another unusual feature of \eqref{eq:logNLS} concerns the effect of
the size of the initial data on the dynamics:
If $u$ solves \eqref{eq:logNLS}, then for all
$k\in \C$, so does
\begin{equation}\label{eq:scaling}
u_k(t,x):=  k u(t,x) e^{-it\lambda\ln|k|^2}.  
\end{equation}
This shows that the size of the initial data alters the dynamics only
through a purely time dependent oscillation, a feature which is
fairly unusual for a nonlinear equation. For $k>0$, we readily compute
\begin{equation*}
  \frac{d}{dk} u_k(t,x) = (1-2it)u(t,x) e^{-it\lambda\ln|k|^2} .
  \end{equation*}
 The above quantity has no limit as $k\to 0$ for $t>0$: the flow map
$u_0\mapsto u(t)$ cannot be $C^1$, whichever function spaces are
considered for $u_0$ and $u(t)$, respectively; it is at most
Lipschitzean.
\smallbreak

The next few sections are dedicated to the analysis of the logarithmic
Schr\"odinger equation \eqref{eq:logNLS}. In Section~\ref{sec:disp},
we will see a first connection with models from compressible fluid mechanics, and
from Section~\ref{sec:transition} to the end of the survey, we will
focus our attention on such 
models and some of their generalizations.

\subsection{Schematic summary: power vs. logarithmic
  nonlinearity} 

In the following tables, we give an overview of the results presented
below, in order to emphasize some differences due to the nature
of the nonlinearity. To avoid unnecessary technical details, we assume
in all cases that the initial datum $u_0$ belongs to $\Sigma$, defined
by
\begin{equation*}
  \Sigma=H^1\cap \F(H^1)=\{f\in H^1(\R^d),\ x\mapsto |x|f(x)\in
  L^2(\R^d)\},
\end{equation*}
that the nonlinearity is $H^1$-subcritical,
$0<\si<\frac{2}{(d-2)_+}$, and do not try to invoke sharp
results. (GWP stands for global well-posedness.)

\bigbreak

\noindent {\bf Case $\lambda>0.$}\\

\hspace{-14pt}
\begin{tabular}{|l|c|c|}
  \hline
  Equation & \eqref{eq:NLS} & \eqref{eq:logNLS}\\
  \hline 
  Nonlinearity & $\lambda |u|^{2\si}u$ & $\lambda \ln(|u|^2)u$\\
  \hline
  GWP in $H^1$ & Yes & Yes \\
  \hline
  Dispersion: $(\delta (t))^{-d/2}$ & $\delta (t)=t$ (at least if $\sigma\ge 2/d$) &
    $\delta (t)=2t\sqrt{\lambda  \ln t}$\\
  \hline
   $\dis\lim_{t\to \infty} (\delta
  (t))^{d/2}|u(t,x\delta (t))|$ & $|v_+|$ for any $v_+\in \Sigma$ (if $\si\ge
                                  2/d$) &
                                          $\dis\frac{\|u_0\|_{L^2}}{\pi^{d/4}}
                                          e^{-|x|^2/2}$ \\
  \hline
Growth of $H^1$ norm& Never & Always\\
  \hline
\end{tabular}
\smallbreak

In connection with the results recalled in Sections~\ref{sec:LS} and
\ref{sec:NLS}, $v_+=\hat u_+$, the Fourier transform of the asymptotic
state $u_+$, which may be any function in $\Sigma$. 
\bigbreak
\noindent {\bf Case $\lambda<0.$}\\

\hspace{-14pt}
\begin{tabular}{|l|c|c|}
  \hline
  Equation & \eqref{eq:NLS} & \eqref{eq:logNLS}\\
  \hline 
  Nonlinearity & $\lambda |u|^{2\si}u$ & $\lambda \ln(|u|^2)u$\\
  \hline
  GWP in $H^1$ & For sure, if $\si<2/d$ & Yes \\
  \hline
  Finite time blow up & Possible if $\si\ge 2/d$
                            & Never\\
  \hline
Dispersion & For small data, if $\si\ge 2/d$ &
    Never \\
  \hline
  Standing waves & Exist & Exist \\
  \hline
  Ground states are orbitally stable & Yes, iff $\si<2/d$ & Yes\\
  \hline
  Breathers & ? & Exist \\
 \hline
  Multisolitons & Exist & Exist \\
  \hline
  Multibreathers & ? & Exist \\
  \hline
\end{tabular}

\section{Cauchy problem}
In view of the expression \eqref{eq:ElogNLS}, the natural energy space is given by
\begin{equation*}
  W:= \Big \{ u\in H^1(\R^d) \, , \,  x\mapsto |u(x)|^2\ln |u(x)|^2\in L^1(\R^d)\Big\} \, . 
\end{equation*}
The Cauchy problem \eqref{eq:logNLS} is indeed solved
in $W$, provided that $\lambda<0$:
\begin{theorem}[\cite{CaHa80}]
Suppose $\lambda <0$ and $u_0\in W$: there exists a unique, global 
  solution $u\in C(\R;W)$ to \eqref{eq:logNLS}. The
  mass~$M(u)$ and the energy $E(u)$ are independent of time. 
\end{theorem}
The proof given in \cite{CaHa80} relies on compactness arguments,
using a regularization of the 
  nonlinearity.
An alternative proof has been proposed more recently by Masayuki
Hayashi \cite{HayashiM2018}, providing 
  the strong convergence of a sequence of approximate solutions. See
  also \cite{GLMN12} for the local Cauchy problem in $H^2$  on bounded
 3D domains.  
We emphasize that the sign of $\lambda$ appears when seeking a
priori estimates: if $\lambda<0$, we have
  \begin{align*}
 0\le    E_+(u(t)) &:= \frac{1}{2}\|\nabla
  u(t)\|_{L^2(\R^d)}^2+\lambda\int_{|u|<1}
  |u(t,x)|^2\ln |u(t,x)|^2dx  \\
&\le E(u_0)\underbrace{-\lambda}_{+|\lambda|} \int_{|u|>1}
  |u(t,x)|^2\ln |u(t,x)|^2dx .
  \end{align*}
Since the logarithm grows slowly,
\begin{align*}
  \int_{|u|>1}
  |u(t,x)|^2\ln |u(t,x)|^2dx &\le C_\eps \int_{|u|>1}
  |u(t,x)|^{2+\eps}dx \\
&\lesssim \|u(t)\|_{L^2(\R^d)}^{2+\eps-\eps
    d/2} \|\nabla u(t)\|_{L^2(\R^d)}^{\eps d /2},
\end{align*}
where we have used Gagliardo-Nirenberg inequalities, for
$\eps<2d/(d-2)_+$. Using the conservation of 
mass, this implies
\begin{equation*}
  E_+(u(t)) \le E(u_0)+C_\eps E_+\(u(t)\)^{\eps d /4}.
\end{equation*}
Therefore, picking $\eps>0$ sufficiently small yields $E_+\in
L^\infty(\R)$, hence (resuming the above inequality) $u\in
L^\infty(\R;W)$. 
\begin{remark}\label{rem:moment-cauchy}
  In the case $\lambda>0$, the same strategy would require the control of
  \begin{equation*}
  \int_{|u|<1}
  |u(t,x)|^2\ln \frac{1}{|u(t,x)|^2}dx \le C_\eps \int_{|u|<1}
  |u(t,x)|^{2-\eps}dx .
\end{equation*}
The above Lebesgue norm involves an index below $2$, and Sobolev
embedding cannot help: we will see that the finiteness of a momentum
in $L^2$ saves the day. 
\end{remark}
Uniqueness follows from the algebraic property discovered in \cite{CaHa80}:
\begin{lemma}[\cite{CaHa80}]\label{lem:unique}
  We have 
  \begin{equation*}
    \left| \IM\left(\left(z_2\ln|z_2|^2
      -z_1\ln|z_1|^2\right)\left(\bar z_2-\bar z_1\right)\right)\right|\le
    4|z_2-z_1|^2 \, ,\quad \forall 
    z_1,z_2\in {\mathbb C} \, .
  \end{equation*}
\end{lemma}
Formally, if $u_1$ and $u_2$ are two solutions of \eqref{eq:logNLS},
$w:=u_2-u_1$ solves
\begin{equation*}
  i\d_t w+\frac{1}{2}\Delta w = \lambda\(u_2\ln|u_2|^2- u_1\ln|u_1|^2\).
\end{equation*}
Multiplying the above equation by $\bar w$, integrating in space and
considering the imaginary part, Lemma~\ref{lem:unique} yields
\begin{equation*}
  \frac{1}{2}\frac{d}{dt}\|w(t)\|_{L^2}^2 \le 4 |\lambda|\|w(t)\|_{L^2}^2 ,
\end{equation*}
and Gronwall lemma provides uniqueness. One has to be cautious though,
this argument is fully justified provided that we know 
$u\in C(\R;L^2)$, a property which is satisfied for $u\in
L^\infty(\R;W)$, as shown by a careful examination of
\eqref{eq:logNLS}; see \cite{CaHa80} for details. 
\smallbreak

This strategy was adapted in \cite{GuLoNi10} to consider the case
$\lambda>0$, under the extra assumption $|x|^{1/2}u_0\in L^2(\R^d)$,
for $d=3$. 
\smallbreak

In \cite{CaGa18}, another compactness method was proposed, consisting
in neutralizing the singularity of the logarithm at the origin: for
$\eps>0$, consider $u^\eps$ solution to 
\begin{equation}
  \label{eq:logNLS-app}
  i\d_t u^\eps +\frac{1}{2} \Delta u ^\eps = \lambda \ln\(\eps+|u^\eps|^2\)u^\eps  ,\quad u^\eps_{\mid t=0} =u_0  .
\end{equation}
For fixed $\eps>0$, the above nonlinearity is smooth and
$L^2$-subcritical, and there exists a unique solution at the
$L^2$-level \cite{TsutsumiL2}. 
Assuming $u_0\in H^1$, $|x|^\alpha u_0\in L^2$ for some $0<\alpha\le
1$, we can prove a priori estimates on bounded time intervals, which
are uniform in $\eps\in ]0,1]$, and infer:
\begin{theorem}[\cite{CaGa18}]\label{theo:GWP}
Let $\lambda\in \R$,  $u_0\in H^1\cap \F(H^\alpha)$ for some
$0<\alpha\le 1$: \eqref{eq:logNLS} has a unique, global 
  solution $u\in L^\infty_{\rm loc}(\R;H^1\cap \F(H^\alpha))$. The
  mass~$M(u)$ and the energy $E(u)$ are independent of time. 
\end{theorem}
Heuristically, the assumption $|x|^\alpha u_0\in L^2$ for some $0<\alpha\le
1$ seems rather natural in view of
Remark~\ref{rem:moment-cauchy}. Indeed,
considering $0<\eta<\frac{4\alpha}{d+2\alpha}$, we have
  \begin{equation}\label{eq:GNdual}
  \int_{{\mathbb R}^d} |u|^{2-\eta} \lesssim \|u\|_{L^2}^{2-\eta-\frac{d \eta}{2\alpha}}
  \left\lVert \lvert x\rvert^\alpha u\right\rVert_{L^2}^{\frac{d \eta}{2\alpha}} .
\end{equation}
This estimate is readily established by using H\"older estimate: let
$p=\frac{2}{2-\eta}$, so its dual exponent is $p'=\frac{2}{\eta}$. Fix
$R>0$, and write
\begin{equation*}
   \int_{|x|<R} |u|^{2-\eta}\le |B(0,R)|^{1/p'}
   \|u\|_{L^2(\R^d)}^{2/p} \lesssim R^{d/p'} \|u\|_{L^2(\R^d)}^{2-\eta}.
 \end{equation*}
 For large $x$, write
 \begin{equation*}
   \int_{|x|>R} |u|^{2-\eta}=  \int_{|x|>R} |x|^{-\beta}|x|^\beta
   |u|^{2-\eta}\lesssim \(\int_R^\infty
   \frac{r^{d-1}}{r^{p'\beta}}dr\)^{1/p'} \left\|
     |x|^{\frac{\beta}{2-\eta}}u\right\|_{L^2(\R^d)} ^{2-\eta}.
 \end{equation*}
 We now choose $\beta$ so that $\frac{\beta}{2-\eta}=\alpha$ and that
the first integral on the right-hand side is finite, $p'\beta>d$. This
means $0<\eta<\frac{4\alpha}{d+2\alpha}$, and optimizing in $R$,
\begin{equation*}
   R^{d/p'} \|u\|_{L^2(\R^d)}^{2-\eta} = R^{d/p'-\beta}\left\|
     |x|^{\frac{\beta}{2-\eta}}u\right\|_{L^2(\R^d)}
   ^{2-\eta}\Longleftrightarrow R^\alpha = \frac{\left\lVert \lvert
       x\rvert^\alpha u\right\rVert_{L^2}}{\|u\|_{L^2}}, 
 \end{equation*}
 we obtain \eqref{eq:GNdual}. Therefore, choosing $0<\eta\ll 1$, we
 guess that any solution in $H^1\cap \F(\H^\alpha)$ is global.
\smallbreak

 The
 complete argument to prove Theorem~\ref{theo:GWP} consists in
 differentiating \eqref{eq:logNLS-app} 
 in space, and using the same $L^2$-estimate as for uniqueness
 (multiply by $\nabla \bar u$, integrate in space, and take the
 imaginary part), to get
 \begin{equation*}
   \frac{1}{2}\frac{d}{dt}\|\nabla u^\eps(t)\|_{L^2}^2 \le
   2|\lambda|\|\nabla u^\eps(t)\|_{L^2}^2 ,
 \end{equation*}
hence a control on bounded time intervals, which is uniform in
$\eps$. To get compactness in space, we compute
\begin{align*}
\frac d{dt} \|\<x\>^\alpha u^\eps(t)\|_{L^2(\R^d)}^2 &= 2\alpha\IM \int  \frac{x
  \cdot \nabla  u_{\varepsilon}  }{  \langle x \rangle^{2-2\alpha}} \,
\overline u_{\varepsilon}  \, dx \\ 
& \le 2\alpha\|\left\langle x\right\rangle^{2\alpha-1}u_\varepsilon(t)
\|_{L^2({\mathbb R}^d)} \|\nabla 
u_\varepsilon(t)  \|_{L^2({\mathbb R}^d)} \\
&\le 2\alpha\|\left\langle x\right\rangle^{\alpha}u_\varepsilon(t)  \|_{L^2({\mathbb R}^d)} \|\nabla
u_\varepsilon(t)  \|_{L^2({\mathbb R}^d)},
\end{align*}
since $\alpha\le 1$, hence a closed system of estimate, uniformly in
$\eps$. We refer to \cite{CaGa18} for the remaining arguments. \\

We recall the notation
\begin{equation*}
  \Sigma=H^1\cap \F(H^1)=\{f\in H^1(\R^d),\ x\mapsto |x|f(x)\in
  L^2(\R^d)\},
\end{equation*}
and keep in mind that if $u_0\in \Sigma$, then \eqref{eq:logNLS} has a
unique solution $u\in L^\infty_{\rm loc}(\R;\Sigma)$.

\begin{remark}[Higher regularity]
  As the nonlinearity $z\mapsto z\ln|z|^2$ has limited regularity, it
  is not obvious to propagate higher $H^s$ regularity in
  \eqref{eq:logNLS}. Typically, 
  \begin{align*}
    \d_j^2 \(u\ln|u|^2\) & = \d_j^2u \ln |u|^2 +\frac{\ln \bar u}{u}(\d_j u
               )^2+4\frac{\ln u}{u}|\d_j u|^2 + \d_j^2u \ln \bar u
                           +\frac{u}{\bar u}\ln u \,\d_j^2\bar u \\
    &\quad-\frac{u}{\bar
               u^2}(\d_j \bar u)^2\ln u,
  \end{align*}
  and it is not clear to propagate the $H^2$ regularity by
  differentiating the equation twice in space. On the other hand, a
  specificity of Schr\"odinger equations is that $H^2$ regularity in
  space can be read from the $L^2$ regularity of $\d_t u$, see
  e.g. \cite[Section~5.3]{CazCourant}. Replacing the  space
  derivatives with time derivative in the above estimates, and
  noticing that if $u_0\in H^2\cap \F (H^\alpha)$ for some $\alpha>0$,
  \begin{equation*}
    i\d_t u_{\mid t=0} = -\frac{1}{2}\Delta u_0 +\lambda
    \ln\(|u_0|^2\) u_0\in L^2(\R^d),
  \end{equation*}
  we infer $\d_t u \in L^\infty_{\rm loc}(\R;L^2(\R^d))$. 
  Theorem~\ref{theo:GWP} implies $u\ln|u|^2\in L^\infty_{\rm
    loc}(\R;L^2(\R^d))$, and then from \eqref{eq:logNLS}, $\Delta u\in
  L^\infty_{\rm loc}(\R;L^2(\R^d))$, hence $u\in   L^\infty_{\rm
    loc}(\R;H^2(\R^d))$. However, propagating $H^3$ regularity (and
  higher)  is  still an open question. 
\end{remark}
\section{Explicit Gaussian solutions}\label{sec:gaussian}
An important feature of \eqref{eq:logNLS}, noticed already in
\cite{BiMy76},  is that Gaussian initial data propagate as Gaussian
solutions. Plugging Gaussian solutions into \eqref{eq:logNLS}, this
property is suggested by the property that in the presence of a
quadratic, possibly time dependent potential in (linear) Schr\"odinger
equations,
\[ i \d_t u + \frac{1}{2}\Delta u =\sum_{j=1}^d \Omega_j(t)x_j^2 u\quad
    ;\quad u_{\mid t=0} = u_0,
    \]
Gaussian   initial data propagate as Gaussian
solutions; see e.g. \cite{Hepp,Hag80,Hag81}. 
\subsection{General computation}

 Suppose $d=1$, and plug $ u(t,x) =
  b(t)e^{-a(t)x^2/2}$ into \eqref{eq:logNLS}: simplifying by
  $e^{-a(t)x^2/2}$, we get
 \begin{equation*}
  i\dot b -i \dot a \frac{x^2}{2}b
  -\frac{ab}{2} +a^2
  \frac{x^2}{2}b = \lambda \( \ln\(|b|^2\) -\(\RE a\) x^2
  \)b ,
\end{equation*}
 hence
\begin{equation*}
  i\dot a -a^2=2\lambda \RE a\, ;\quad i\dot b -\frac{ab}{2} = 
\lambda b \ln\(|b|^2\).
\end{equation*}
We can express $b$ as a function of $a$:
\begin{equation}\label{eq:formule-b}
   b(t) = b_0 \exp \({-i\lambda t \ln\(|b_0|^2\)-\frac{i}{2}A(t)
    -i\lambda \IM\int_0^tA(s)sds} \)\, ,
\end{equation}
where we have set $\dis A(t):=\int_0^t a(s)ds$. 
 So we focus on
\begin{equation*}
  i\dot a -a^2=2\lambda \RE a,\quad a_{\mid t=0}=a_0=\alpha_0+i\beta_0.
\end{equation*}
We seek $a$ of the form $\dis 
a = -i\frac{\dot \omega}{\omega}$.  This yields $\dis
 \ddot \omega = 2\lambda \omega\IM\frac{\dot \omega}{\omega}$. \\
Introducing a polar decomposition $\omega=r e^{i\theta}$, we find
\begin{equation*}
   \ddot r -(\dot \theta)^2r = 2\lambda r \dot \theta \, ;\quad \ddot
   \theta r +2\dot \theta \dot r=0  . 
\end{equation*}
Notice that
\[
  \dot \theta_{\mid t = 0} = \alpha_0 \, , \quad \left(\frac{\dot r}r\right)_{\mid t = 0} = -\beta_0 \, .
\]
We  decide~$r(0)  =1$, so $
  \dot \theta(0)=\RE a_0=\alpha_0$ and $ \dot r(0)=-\IM a_0=-\beta_0.$
Note
\begin{equation*}
 \frac{d}{dt}\( r^2\dot \theta\) = r\( 2\dot r \dot
 \theta+r\ddot \theta\)=0 \, , 
\end{equation*}
and we can express the problem in terms
of $r$ only: 
\begin{equation}\label{eq:equa-diff-r}
 a (t)= \frac{\alpha_0}{r(t)^2}-i\frac{\dot r(t)}{r(t)},
\quad   \ddot r =\frac{\alpha_0^2}{r^3} + 2\lambda
   \frac{\alpha_0}{r},\quad r(0)=1 \,,\ \, \dot r(0)= -\beta_0  \,. 
\end{equation}
Multiply by $\dot r$ and integrate:
\begin{equation}\label{eq:dotr}
  (\dot r)^2 = \beta_0^2 +\alpha_0^2 -\frac{\alpha_0^2}{r^2}+4\lambda
  \alpha_0\ln |r|. 
\end{equation}
Cauchy-Lipschitz theorem yields the existence of a unique local
solution. The obstruction to the existence of a global solution is the
possibility of $r$ going to zero. 
Supposing $r\to 0$ leads to a contradiction in \eqref{eq:dotr}, so
there exists 
\[r(t)\ge \delta>0,\quad t\ge 0,\]
and the solution is global in time, and smooth.
\begin{remark}[Decay rate]
  In the above computations, the formula \eqref{eq:formule-b} shows
  that the decay rate is given by
  \begin{equation*}
    |b(t)| = |b_0|\exp\( \frac{1}{2}\IM A(t)\)= \frac{|b_0|}{\sqrt{r(t)}},
  \end{equation*}
  where the last identity follows from \eqref{eq:equa-diff-r}. 
\end{remark}
\begin{remark}[Linear Schr\"odinger equation]
  In the case $\lambda=0$ (linear case), the equation for $r$ reads
  \[ \ddot r =\frac{\alpha_0^2}{r^3} ,\quad r(0)=1 \,,\ \, \dot r(0)=
    -\beta_0  \,. \]
  The solution is given by
  \[ r(t) =\sqrt{1+t^2(\alpha_0^2+\beta_0^2)-2t\beta_0},\]
  which is a rather indirect way to recover the formula presented in
  Example~\ref{ex:gaussian}. 
\end{remark}
\subsection{Nondispersive case: $\lambda<0$}

Suppose $\lambda<0$: in view of \eqref{eq:dotr}, $r$ is bounded. 
Standard techniques in the study of ordinary differential equations
show that
every solution is periodic in time.
The relation \eqref{eq:dotr} defines the potential
energy (see e.g. \cite{Arn06})
\begin{equation*}
  U(r)= -\frac{1}{2}\beta_0^2 -\frac{\alpha_0^2}{2}\(1-\frac{1}{r^2}\)
  -2\lambda \alpha_0 \ln |r|= -\frac{1}{2}\beta_0^2 -\frac{\alpha_0^2}{2}\(1-\frac{1}{r^2}\)
  +2|\lambda|\alpha_0 \ln |r|.
\end{equation*}
We check that $U$ is decreasing on
$(0,\sqrt{\alpha_0/2|\lambda|}]$, and increasing on
$[\sqrt{\alpha_0/2|\lambda|},\infty)$. The minimum is given by
\begin{equation*}
  U_{\rm min}  =  -\frac{1}{2}\beta_0^2
                +\frac{\alpha_0^2}{2}\(x-1-x\ln x\)\Big|_{x =
                \frac{2|\lambda|}{\alpha_0}}\le 0,
\end{equation*}
in view of the property
\begin{equation*}
  x-1-x\ln x\le 0,\quad \forall x>0.
\end{equation*}
We have $U_{\rm min} <0$ unless $\beta_0=\dot r(0)=0$ and
$\alpha_0=2|\lambda|$, the only case where $U_{\rm min} =0$.  See
Figure~\ref{fig:potentiel}. Note that
the case $\beta_0=0$ and
$\alpha_0=2|\lambda|$ is degenerate, in the sense that we then have
$r(t)\equiv 1$.
\begin{figure}[ht]
\begin{center}
\rotatebox{0}{\resizebox{!}{0.5\linewidth}{%
   \includegraphics{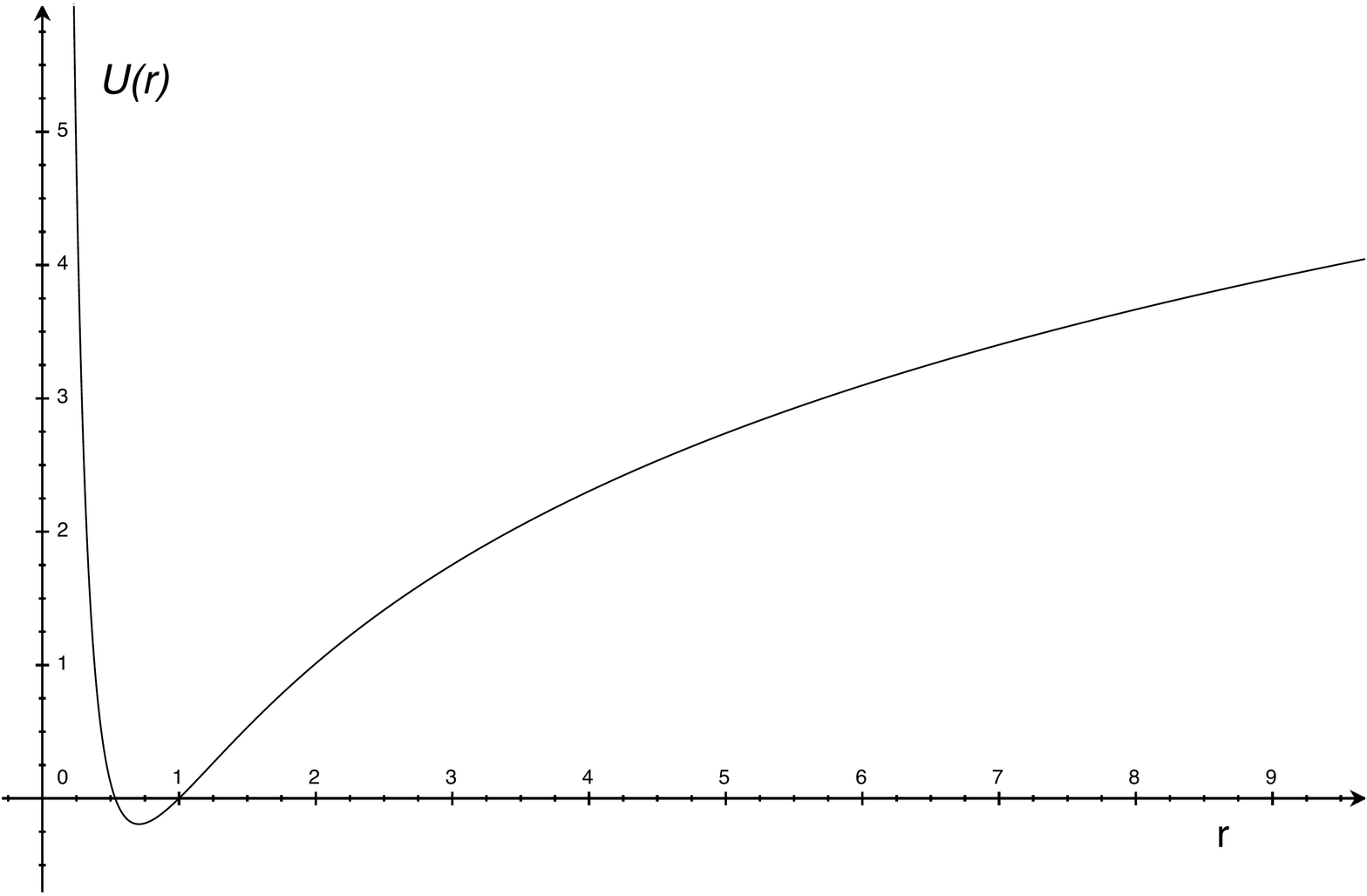}}}
\end{center}
\caption{Potential $U$ in the case $a_0=1$ and $\lambda=-1$.} 
\label{fig:potentiel}
\end{figure}

For every energy $E>U_{\rm min}$, the equation
$U(r)=E$ has two distinct solutions. We infer (see
e.g. \cite{Arn06}) that all the solutions to \eqref{eq:equa-diff-r}
are periodic, and the half-period is given by
\begin{equation*}
  \frac{T}{2}= \int_{r_*}^{r^*}\frac{dr}{\sqrt{E-U(r)}},
\end{equation*}
where $r_*<r^*$ are the two above mentionned solutions.
Therefore, the amplitude of the corresponding
(Gaussian) solution $u$ is time-periodic, and we naturally call such
solutions \emph{breathers}; see \cite{FeDCDS} for more
details. Note however that only the amplitude  is
  periodic, not the solution $u$ itself: in view of
  \eqref{eq:formule-b}, $|b(t)|$ is periodic, but not 
  $b(t)$, unless $|b_0|=1$ and $\IM A\equiv -0$.
\smallbreak

As pointed out above, the relations $\beta_0=0$, $\alpha_0=-2\lambda$
imply $r\equiv 1$. This provides a stationary solution,
\begin{equation*}
  u(t,x) =  \sqrt  e \, e^{\lambda |x|^2}.
\end{equation*}
In view of the remarkable scaling property
\eqref{eq:scaling}, and of 
the tensorization property discussed in the introduction, we
infer the existence of 
infinitely many standing waves in $\R^d$ (we had assumed $d=1$ so far),
parametrized by $\omega\in \R$ (related to $k\in \R$ in \eqref{eq:scaling}
through the relation $\omega = -\lambda \ln (k^2)$),
\begin{equation*}
  u_\omega(t,x) = e^{i\omega t} e^{\frac{d}{2} -
    \frac{\omega}{2\lambda}}e^{\lambda |x|^2}. 
\end{equation*}
These standing waves, discovered in \cite{BiMy76} (see also
\cite{BiMy79}), are known as \emph{Gaussons}. Given $\omega\in \R$ is
arbitrary, for each prescribed mass $M$, there exists a (unique)
Gausson whose mass is equal to $M$. 
\subsection{Dispersive case: $\lambda>0$}

In this case, $r$ is strictly (but not uniformly) convex. We can prove
the following: there exists $T\ge 0$ such that for $t\ge T$, $\ddot r>0$, and $r(t)\to \infty$ as $t\to
\infty$. Therefore, the dynamics is expected to be well approximated by
\begin{equation*}
  \ddot r_{\rm eff} = 
   \frac{2\lambda\alpha_0}{r_{\rm eff}}\quad (\alpha_0>0).
\end{equation*}
Multiply by $\dot  r_{\rm eff} $ and integrate: since (at least) for
$t\gg 1$, $\dot  r_{\rm eff} >0$, we find
\begin{equation*}
  \dot r_{\rm eff}= \sqrt{C_0 + 4\lambda \alpha_0\ln r_{\rm eff}},
\end{equation*}
Separate variables:
\begin{equation*}
  \int^{r_{\rm eff}} \frac{dz}{\sqrt{C_0 + 4\lambda \alpha_0\ln z}}=t-T.
\end{equation*}
Set $y= \sqrt{C_0 + 4\lambda \alpha_0\ln z}$. The left hand
side becomes
\begin{equation*}
  \frac{1}{2\lambda \alpha_0}\int^Y e^{(y^2-C_0)/(4\lambda \alpha_0)}dy.
\end{equation*}
The asymptotics of Dawson function (see e.g. \cite{AbSt64}) yields,
for any $x_0\in \R$,
\begin{equation*}
  \int_{x_0}^x e^{y^2} dy\Eq x \infty \frac{1}{2x }e^{x^2}\Longrightarrow
 \frac{r_{\rm eff}}{\sqrt{C_0 + 4\lambda \alpha_0\ln
       r_{\rm eff}}} \Eq t \infty t.
\end{equation*}
Since $r_{\rm eff}\to \infty$,
 $\dis  \frac{r_{\rm eff}}{\sqrt{4\lambda \alpha_0\ln
       r_{\rm eff}}} \Eq t \infty t,$ 
hence
\begin{equation*}
  r_{\rm eff}(t)\Eq t \infty 2t\sqrt{\lambda \alpha_0\ln t}.
\end{equation*}
We note that $C_0$ has disappeared, \emph{at leading order}. Thus, all the
Gaussian solutions have the same asymptotic profile, with a
nonstandard dispersion (enhanced compared to the standard one, by a
logarithmic factor). Up to scaling the solution and changing
initial data, we can simply 
consider $\dis \ddot \tau= 
  \frac{2\lambda}{\tau}$, with $\tau(0)=1$ and $\dot \tau(0)=0$. This
  yields a uniform dispersion in the case of Gaussian data. We will
  see that this dispersion is actually completely general, in the case
  $\lambda>0$. 

\section{Solitons}

In this section, we always assume $\lambda<0$. The formal discussion presented in Section~\ref{sec:logNLS} suggests that when
$\lambda<0$, dispersion is impossible. This has been proven rigorously
by variational arguments by Thierry Cazenave:
 \begin{lemma}[\cite{Caz83}]
  Let $\lambda<0$ and $k<\infty$ such that
\[
L_k := \{u\in W, \ \|u\|_{L^2(\R^d)}=1,\ E(u)\le k\}\not
  =\emptyset \, .\]
 Then
 $\dis  \inf_{{u\in L_k}\atop {1\le p\le \infty}}\|u\|_{L^p(\R^d)}>0 .$
\end{lemma}
In view of the conservation of the mass and the energy, this implies
that no solution is dispersive in the case $\lambda<0$ (otherwise an
$L^p$-norm would go to zero for some $p>2$). 
\subsection{Gaussons}
As we have seen above, we have explicit standing waves, called
Gaussons, given by the formula
\begin{equation*}
  u_\omega(t,x) = e^{i\omega t} e^{\frac{d}{2} -
    \frac{\omega}{2\lambda}}e^{\lambda |x|^2}.
\end{equation*}
Due to the Galilean invariance \eqref{eq:galilee}, such solutions are
not asymptotically 
stable: multiplying $u_\omega(0,x)$ by $e^{i{\mathbf v}\cdot x}$, for some small
${\mathbf v}$, is a small perturbation in $H^1$, but the drift in space, given
by ${\mathbf v} t$, shows that the corresponding solution does not remain close
to $u_\omega$ for all time. Even in the radial case (where Galilean
invariance is absent), it is necessary to take phase shifts into
account, as noticed in \cite{Caz83} by an explicit example on Gaussian
data, which is related to \eqref{eq:scaling}: for $|\omega-\omega'|\ll
1$, $u_\omega(0,\cdot)$ and $u_{\omega'}(0,\cdot)$ are initial data
which are close to each other in $H^1(\R^d)$, but of course the
corresponding solutions $u_\omega$ and $u_{\omega'}$ present a
non-negligible phase shift e.g. for $(\omega-\omega')t = \pi$. On the other hand, Gaussons 
are \emph{orbitally stable}: this was proven in \cite{Caz83} for the
radial case, and in \cite{Ar16} for
the general case (the key variational step there is based on the
logarithmic Sobolev inequality). 

\begin{theorem}[\cite{Caz83,Ar16}]
  Let $\lambda<0$ and $\omega\in \R$. Set
  \begin{equation*}
    \phi_\omega(x) =  e^{\frac{d}{2} -
    \frac{\omega}{2\lambda}}e^{\lambda |x|^2}.
  \end{equation*}
  For any $\eps>0$, there exists $\eta>0$ such that if $u_0\in W$
satisfies $\|u_0-\phi_\omega\|_W<\eta$, then the solution $u$ to
\eqref{eq:logNLS} exists for all $t\in \R$, and
  \begin{equation*}
    \sup_{t\in \R}\inf_{\theta\in \R}\inf_{y\in \R^d}\|u(t)
    -e^{i\theta}\phi_\omega(\cdot-y)\|_W<\eps.
  \end{equation*}
\end{theorem}

\subsection{Superposition}

Numerical experiments reported in \cite{BCST19b} reveal dynamical
properties which are fairly different from the phenomena observed in
the case of a power-like nonlinearity, when solitons interact. In
particular, two Gaussons centered far apart seem motionless over a
long time of simulation. Each Gausson solves \eqref{eq:logNLS}
exactly, but the equation being nonlinear, the sum of two Gaussons
does not solve \eqref{eq:logNLS}: there seems to be a rather strong
superposition principle however. This was proven rigorously in \cite{FeDCDS},
in a more general framework: starting from finitely many initial
Gaussians (not necessarily Gaussons) with pairwise distances of order at least $
1/\eps$ (for $0<\eps\ll 1$), the solution of \eqref{eq:logNLS} is well
approximated by 
the sum of the 
corresponding solutions (computed in Section~\ref{sec:gaussian}), over a time $o(\eps^{-2})$. More precisely, the error is of
order $e^{c_0t-c_1/\eps^2}$ for some constants $c_0,c_1>0$ expressed explicitly
in \cite{FeDCDS}.
\smallbreak

At this stage, we emphasize an aspect which is crucial in the next
section too: the logarithmic nonlinearity is not Lipschitz continuous
at the origin, and in particular any linearization process becomes
very delicate. To overcome this difficulty, the strategy employed in
\cite{FeDCDS} consists in establishing fine properties of the
logarithmic nonlinearity. Typically, the nonlinearity
$F(z):=z\ln|z|^2$ satisfies, for $|z|,|z'|\le 1$, $z\not =0$,
\begin{equation}\label{eq:GuillaumeDCDS}
  |F(z)-F(z')|\le |z-z'|\(6-\ln|z|^2\).
\end{equation}
The interest of this estimate lies in the fact that it is not
symmetric in $(z,z')$. This is crucial in order to estimate
the source term in the equation solved by the difference between the
exact solution and the sum of individual Gaussian solutions, which is
of the form
\begin{equation*}
  u\ln |u|^2 - \sum_{j=1}^N g_j\ln|g_j|^2.
\end{equation*}
\subsection{Multigaussons}

Again, we do not make complete statements here, and try to give a
flavor of the corresponding result. 
Using the Galilean invariance, introduce, for some $k\ge 1$
\begin{equation*}
  G_k= \sum_{j=1}^k \Gamma_j(t,x), \quad
  \mathbb B_k =\sum_{j=1}^k B_j(t,x),
\end{equation*}
where the $\Gamma_j$'s are Gaussons associated with pairwise different velocities
${\mathbf v}_j$, and the $B_j$'s are (more general) breathers associated with
pairwise different velocities ${\mathbf v}_j$. Unlike in \cite{CoteLeCoz2011}, it is not
assumed that the relative velocities ${\mathbf v}_j-{\mathbf v}_k$, $j\not = k$, are large. As evoked above, linearization
is a delicate process here, and so, even though the following
statement is reminiscent of \cite{MartelMerle2006} or \cite{Schuur1986} (for the
modified KdV equation, an equation which is integrable), the
approach must be adapted:
\begin{theorem}[\cite{FeAIHP}]
  Let $d\ge 1$ and $\lambda<0$. 
  \begin{itemize}
  \item Multibreathers: there exists a unique solution $u\in C_b(\R;W)$ to
    \eqref{eq:logNLS}, and $c,C>0$ such that
    \begin{equation*}
      \|u(t)-\mathbb B_k(t)\|_{L^2(\R^d)}\le C e^{-ct^2},\quad t\ge 0.
    \end{equation*}
   \item Multigaussons:  there exists a unique solution $u\in C_b(\R;\Sigma)$ to
    \eqref{eq:logNLS}, $c,C>0$ such that
    \begin{equation*}
      \|u(t)-G_k(t)\|_{\Sigma}\le C e^{-ct^2},\quad t\ge 0.
    \end{equation*}
  \end{itemize}
\end{theorem}
We emphasize a few aspects, and refer to \cite{FeAIHP} for details:
\begin{itemize}
\item The construction is based on compactness techniques, as introduced in
  \cite{MartelMerle2006}.
 \item The linearized operator around the Gausson seems to be nice, as
   it is a 
   harmonic oscillator, whose eigenproperties are very well known. However, the logarithm is singular at zero,
   and so linearizing becomes a delicate matter. Like in the previous
   section, a clever use of \eqref{eq:GuillaumeDCDS} saves the day.
 \item The proof uses localized energy functionals  involving a
   linearized functional, 
   which is not the linearized energy. 
\end{itemize}
\section{Dispersive case}\label{sec:disp}

We now assume $\lambda>0$, and focus on the following result:

\begin{theorem}[\cite{CaGa18}]\label{theo:logNLSdisp}
  Let $\lambda>0$. Introduce the solution $\tau\in C^\infty(\R)$ to
  \begin{equation}\label{eq:tau-libre}
  \ddot \tau = \frac{2\lambda }{\tau} \, ,\quad \tau(0)=1\, ,\quad \dot
  \tau(0)=0\, .
\end{equation}
Then, as
$t\to \infty$, 
$\tau(t)\sim 2t \sqrt{\lambda \ln t}$ and $ \dot
  \tau(t)\sim 2\sqrt{\lambda\ln  t} $. For $u_0\in
  \Sigma\setminus\{0\}$, \eqref{eq:logNLS} 
  has a unique solution $u\in L^\infty_{\rm
    loc}(\R;\Sigma))$. 
Introduce
$\gamma(x):=e^{-|x|^2/2}$, and 
 rescale the solution  to $v=v(t,y)$ by setting
\begin{equation}
  \label{eq:uvDMJ}
  u(t,x)
  =\frac{1}{\tau(t)^{d/2}}v\left(t,\frac{x}{\tau(t)}\right)
\frac{\|u_0\|_{L^2({\mathbb R}^d)}}{\|\gamma\|_{L^2({\mathbb R}^d)}} 
\exp \Big({i\frac{\dot\tau(t)}{\tau(t)}\frac{|x|^2}{2}} \Big) . 
\end{equation}
Then we have
\begin{equation}
  \label{eq:apv}
  \sup_{t\ge 0}\left(\int_{{\mathbb R}^d}\left(1+|y|^2+\left|\ln
    |v(t,y)|^2\right|\right)|v(t,y)|^2dy +\frac{1}{\tau(t)^2}\|\nabla_y
  v(t)\|^2_{L^2}\right)<\infty,
\end{equation}
\begin{equation}\label{eq:moments}
   \int_{{\mathbb R}^d}
  \begin{pmatrix}
    1\\
y\\
|y|^2
  \end{pmatrix}
|v(t,y)|^2dy\Tend t \infty 
 \int_{{\mathbb R}^d}
  \begin{pmatrix}
    1\\
y\\
|y|^2
  \end{pmatrix}
  \gamma^2(y)dy ,
\end{equation}
and
\begin{equation*}
  |v(t,\cdot)|^2 \mathop{\rightharpoonup}\limits_{t\to \infty}
  \gamma^2 
\quad  \text{weakly in }L^1({\mathbb R}^d)  . 
\end{equation*}
\end{theorem}
The last bound in \eqref{eq:apv} shows that the phase introduced in
\eqref{eq:uvDMJ} incorporates the main oscillations in the large time
limit: since $\dot\tau(t)/\tau(t)\sim 1/t$ as $t\to \infty$, we
recover the same oscillation (at leading order) as for the linear
Schr\"odinger equation, see Section~\ref{sec:LS}. On the other hand,
the dispersive rate is different: the boundedness of the momentum of
$v$ shows that $v$ 
is not dispersive, and the factor $t$ present in the expression of
$A(t)$ in Section~\ref{sec:LS}  has been replaced by
$\tau(t)$. The dispersion of $u$ is thus enhanced by a
logarithmic factor, compared to the standard dispersion. Finally,
$|v(t,\cdot)|^2$ has a universal limit, which is reminiscent of the
heat equation rather than of the Schr\"odinger equation.
\smallbreak

As a consequence, the Sobolev norms of every nontrivial solutions are
unbounded, providing a precise answer to a question asked in
\cite{Bo96} regarding the growth of Sobolev norms for Hamiltonian
nonlinear dispersive equations (see also e.g. \cite{Iturbulent,GG17,GLPR18,GHP16,HPTV15,ST21}): 

  \begin{corollary}\label{cor:growth}
  Let $u_0\in \Sigma\setminus\{0\}$, and
  $0<s\le 1$. As $t\to \infty$,
  \begin{equation*}
    \(\ln t\)^{s/2}\lesssim \|u(t)\|_{\dot H^s(\R^d)}\lesssim \(\ln t\)^{s/2},
  \end{equation*}
where $\dot H^s(\R^d)$ denotes the standard homogeneous Sobolev space.
\end{corollary}
\begin{proof}[Proof in the case $s=1$] Differentiate \eqref{eq:uvDMJ}
  with respect to $x$:
\begin{align*}
 &\nabla u(t,x)
  =\frac{1}{\tau(t)^{d/2}}\nabla_x\(v\(t,\frac{x}{\tau(t)}\)
e^{i\frac{\dot\tau(t)}{\tau(t)}\frac{|x|^2}{2}}\)\\
& = \underbrace{\frac{1}{\tau(t)}
    \frac{1}{\tau(t)^{d/2}}\nabla_y v\(t,\frac{x}{\tau(t)}\)e^{i\frac{\dot\tau(t)}{\tau(t)}\frac{|x|^2}{2}}
}_{\|\cdot \|_{L^2} = \frac{1}{\tau}\|\nabla
  v\|_{L^2}=\O(1). }+ \underbrace{i\dot
\tau \frac{1}{\tau(t)^{d/2}}\frac{x}{\tau} v\(t,\frac{x}{\tau(t)}\)
e^{i\frac{\dot\tau(t)}{\tau(t)}\frac{|x|^2}{2}}}_{\|\cdot \|_{L^2} =
\dot \tau \|yv\|_{L^2}\sim \dot \tau\|y\gamma\|_{L^2}\approx
\sqrt{\ln  t} },
\end{align*}
where we have used \eqref{eq:apv} to control the first term, and
\eqref{eq:moments} to show that the last factor is indeed of order
$\dot \tau$. 
 \end{proof}
 \begin{remark}
   For the case $0<s<1$, we refer to \cite{CaGa18}. Essentially, the
   proof relies on \cite[Lemma~5.1]{ACMA}, which states the following
   (we simplify the original statement, which  incorporates a
   semiclassical parameter): there exists a constant $K$ such that for
   all $s\in [0,1]$, all $u\in H^1(\R^d)$ and all $w\in
   W^{1,\infty}(\R^d)$,
   \begin{equation*}
\| |w|^s u\|_{L^2}\le \| |D_{x}|^s u\|_{L^2}+\| (\nabla -i
w)u\|_{L^2}^{s} 
\| u\|_{L^2}^{1-s}+
K \(1+\| \nabla w\|_{L^\infty}\)\| u\|_{L^2}.
\end{equation*}
We then apply this inequality with $w$ the gradient of the quadratic
oscillation in \eqref{eq:apv}, $w(t,x) =
\frac{\dot\tau(t)}{\tau(t)}x$. 
 \end{remark}

 \begin{remark}
   As pointed out by the editorial board, the function $\tau$, its
   asymptotic behavior, and the rescaling \eqref{eq:uvDMJ} were
   already present in \cite{CidDolbeault-p}, a reference we were not
   aware of. 
 \end{remark}

 \begin{remark}
   In the case of a defocusing power nonlinearity, \eqref{eq:NLS} with
   $\lambda>0$, the conservation of the energy implies that the
   $H^1$-norm of $u$ is uniformly bounded in time, unlike in
   Corollary~\ref{cor:growth}. Moreover, when $\si\ge 2/d$ is an
   integer, and $u\in \Sch(\R^d)$,  all the Sobolev norms $\|u(t)\|_{H^s}$ are bounded. This result is
   natural, since in that case, $u$ is asymptotically linear and the
   linear flow $e^{i\frac{t}{2}\Delta}$ preserves the Sobolev norms
   $H^s$ (see e.g. \cite{Ca11}). 
 \end{remark}
 
\begin{remark}\label{rem:perturb}
  These results remain valid when the logarithmic nonlinearity is
  perturbed by an energy-subcritical, defocusing powerlike
  nonlinearity,
  \begin{equation*}
  i\d_t u +\frac{1}{2} \Delta u =\lambda \ln\(|u|^2\)u  +\mu|u|^{2\si}u,\quad u_{\mid t=0} =u_0  ,
\end{equation*}
with $\mu>0$ and $0<\si<\frac{2}{(d-2)_+}$. Surprisingly enough, the
logarithmic nonlinearity is thus the stronger in the above equation.
\end{remark}

\subsection{Elements of proof}\label{sec:elements}

\subsubsection{A priori estimates}

The key step is to change the unknown function in order to get
coercivity. The change of unknown function is motivated by the
explicit computations in the Gaussian case: at leading order, all the
Gaussian solutions have the same dispersion, the same oscillations, and
the same asymptotic profile. Theorem~\ref{theo:logNLSdisp} states that
these three properties are shared by \emph{all} solutions.

Direct computations show that $v=v(t,y)$, given by \eqref{eq:uvDMJ}, solves
\begin{equation*}
 i{\partial}_t v +\frac{1}{2\tau(t)^2}\Delta_y  v = \lambda  v\ln\left\lvert
    \frac{v}{\gamma}\right\rvert^2-\lambda d v\ln \tau
+2\lambda
  v\ln\left(\frac{\|u_0\|_{L^2({\mathbb R}^d)}}{\|\gamma\|_{L^2({\mathbb R}^d)}} \right)
  \, , 
\end{equation*}
where we recall that $\gamma(y)=e^{-|y|^2/2}$, and the initial datum
for $v$ is
\begin{equation*}
v_{\mid t=0}=v_0:=\frac{\|\gamma\|_{L^2({\mathbb R}^d)}}{\|u_0\|_{L^2({\mathbb R}^d)}} u_0 .
\end{equation*}
 Replacing~$v$ with~$ve^{-i\theta(t)}$ for
\[
\theta(t) =
\lambda d\int_0^t\ln \tau(s)ds -2\lambda
 t \ln(\|u_0\|_{L^2}/\|\gamma\|_{L^2}),
\]
a change of unknown function which does not affect the conclusions of
Theorem~\ref{theo:logNLSdisp}, 
 we may assume that the
last two terms are absent, and we focus our attention on
\begin{equation}
  \label{eq:v}
  i{\partial}_t v +\frac{1}{2\tau(t)^2}\Delta_y  v = \lambda v\ln\left\lvert
    \frac{v}{\gamma}\right\rvert^2,\quad 
v_{\mid t=0}=v_0 \, .
\end{equation}
The above equation is still Hamiltonian: introduce
\begin{equation*}
 \mathcal E (t):= \IM\int_{{\mathbb R}^d} \bar v(t,y){\partial}_t v(t,y)dy= \mathcal
 E_{\rm kin} (t)+\lambda \mathcal E_{\rm ent}(t) \, ,
\end{equation*}
where
\begin{equation*}
  \mathcal E_{\rm kin}(t):= \frac{1}{2\tau(t)^2}\|\nabla_y v(t)\|_{L^2}^2  
\end{equation*}
is the (modified) kinetic energy and
\begin{equation*}
 \mathcal  E_{\rm ent} (t):= \int_{{\mathbb R}^d} |v(t,y)|^2 \ln\left\lvert
    \frac{v(t,y)}{\gamma(y)}\right\rvert^2dy = \int_{{\mathbb R}^d}
  |v(t,y)|^2 \ln\left\lvert 
    {v(t,y)}\right\rvert^2dy + \int_{{\mathbb R}^d} |y|^2 |v(t,y)|^2 dy 
\end{equation*}
is a relative entropy.  
 Direct computations yield
\begin{equation}\label{eq:E}
  \dot {\mathcal  E} = -2\frac{\dot \tau}{\tau}\mathcal E_{\rm kin} \, .
\end{equation}
\begin{remark}
  The Csisz\'ar-Kullback inequality reads (see
e.g. \cite[Th.~8.2.7]{LogSob}), for $f,g\ge 0$ with $\int f=\int g$,
\begin{equation*}
  \|f-g\|_{L^1({\mathbb R}^d)}^2\le 2 \|f\|_{L^1({\mathbb R}^d)}\int
  f(x)\ln \left(\frac{f(x)}{g(x)}\right) dx. 
\end{equation*}
Since $|v|^2$ and $\gamma^2$ have the same $L^1$-norm, ${\mathcal
  E}_{\rm ent}\ge 0$: we will not actually use this piece of
information, but this shows that if we could prove ${\mathcal
  E}_{\rm ent}(t)\to 0$ as $t\to \infty$ (which is established in the
case of Gaussian initial data), then the weak convergence in the last
point of Theorem~\ref{theo:logNLSdisp} would become a strong
convergence. 
\end{remark}
The following lemma resumes \eqref{eq:apv}, and contains an extra
integrability property:
\begin{lemma}\label{lem:APv}
  Under the assumptions of Theorem~\ref{theo:logNLSdisp}, 
\begin{equation*}
 \sup_{t\ge 0}\left(\int_{{\mathbb R}^d}\left(1+|y|^2+\left|\ln
    |v(t,y)|^2\right|\right)|v(t,y)|^2dy +\frac{1}{\tau(t)^2}\|\nabla_y
  v(t)\|^2_{L^2({\mathbb R}^d)}\right)<\infty
\end{equation*} 
and
\begin{equation}\label{eq:integralkin}
 \int_0^\infty \frac {\dot \tau(t)}{\tau^3(t)}\|\nabla_y
  v(t)\|^2_{L^2({\mathbb R}^d)} dt<\infty.
\end{equation}
\end{lemma}
\begin{proof}
  Write the pseudo-energy $\mathcal E$  as $\mathcal E=\mathcal
  E_++\mathcal E_-$, where $\mathcal E_+$ gathers the positive terms of
  $\mathcal E$, 
  \begin{equation*}
    \mathcal E_+(t) = \frac{1}{2\tau(t)^2}\|\nabla_y v(t)\|_{L^2}^2  +
    \lambda \int_{|v|>1} |v|^2 \ln|v|^2 +\lambda \int  _{{\mathbb
        R}^d} |y|^2|v|^2 ,
  \end{equation*}
  and
  \begin{equation*}
     \mathcal E_-(t) = 
    \lambda \int_{|v|<1} |v|^2 \ln|v|^2\le 0.
  \end{equation*}
  Since $\mathcal E$ is nonincreasing,
  \begin{equation*}
    \mathcal E_+(t) \le \mathcal E(0)- \mathcal E_-(t) \le \mathcal
    E(0)+ C_\eps \int_{|v|<1} |v|^{2-\eps}\le \mathcal
    E(0)+ C_\eps \int_{\R^d} |v|^{2-\eps},
  \end{equation*}
  for any $0<\eps<2$. Recalling \eqref{eq:GNdual} (with $\alpha=1$), and noting that $\|v(t)\|_{L^2}=\|v(0)\|_{L^2} (=\|\gamma\|_{L^2})$, we
obtain a control of the form
\begin{equation*}
   \mathcal E_+(t) \le \mathcal E(0)+ C \mathcal E_+(t)^{d \epsilon/2} ,
 \end{equation*}
 hence $\mathcal E_+\in L^\infty(\R_+)$ by picking $\eps>0$
 sufficiently small. Then $\mathcal E_-\in L^\infty(\R_+)$, hence
 $\mathcal E\in L^\infty(\R_+)$, and \eqref{eq:integralkin} by just
 saying that $\dot {\mathcal E}$ is integrable. 
\end{proof}

\subsubsection{Center of mass}
Adapting the computation of \cite{Ehrenfest}, introduce
\begin{equation*}
  I_1(t) := \IM \int_{\R^d} \overline v (t,y)\nabla_y v(t,y)dy \, ,\quad I_2(t) :=
  \int_{\R^d} y|v(t,y)|^2dy \, . 
\end{equation*}
We compute:
\begin{equation*}
  \dot I_1= -2\lambda I_2 \,,\qquad \dot I_2 = \frac{1}{\tau(t)^2}I_1 \, .
\end{equation*}
Set $\tilde I_2=\tau I_2$:  $\ddot {\tilde I}_2=0$, hence
\begin{equation*}
  I_2(t) = \frac{1}{\tau(t)}\(\dot
 {\tilde I}_2(0)t+\tilde I_2(0)\)=\frac{1}{\tau(t)}\(-I_1(0)t+I_2(0)\)=\O\(
 \frac{1}{\sqrt{\ln t}} \)\, .
\end{equation*}
In particular,
$\dis  \int_{\R^d} y|v(t,y)|^2dy\Tend t \infty 0 =
\int_{\R^d}y\gamma(y)^2dy.$
If $I_1(0)\not =0$, we also have
\begin{equation*}
  I_1(t)\Eq t \infty c\frac{t}{\sqrt{\ln t}},
  \end{equation*}
  while if $I_1(0)=0\not =I_2(0)$,
  \begin{equation*}
  I_1(t)\Eq t \infty \tilde c\sqrt{\ln t}.
\end{equation*}
\begin{remark}
  In view of Cauchy-Schwarz inequality,
\begin{equation*}
  |I_1(t)|\le \|v\|_{L^2}\|\nabla_y v\|_{L^2} =
  \|\gamma\|_{L^2}\|\nabla_y v\|_{L^2} . 
\end{equation*}
So unless the initial data are centered in zero in phase space
($I_1(0)=I_2(0)=0$), 
\begin{equation*}
  \|\nabla_y v(t)\|_{L^2} \Tend t \infty \infty,
\end{equation*}
suggesting that $v$ is rapidly oscillatory: in general,
\eqref{eq:uvDMJ} filters out the \emph{leading order} oscillations
only, in the limit $t\to\infty$. A careful examination of the
computations in the Gaussian case leads to the same conclusion. This
explains why, in Theorem~\ref{theo:logNLSdisp}, the main results
concern the \emph{modulus} of $v$, and no other quantity (that would
involve the argument of $v$). 
\end{remark}

\subsubsection{Second order momentum}

Introduce
$\displaystyle A = \IM
    \int v\, y \cdot\nabla_y \bar v.$
The estimate \eqref{eq:apv} and Cauchy-Schwarz inequality yield:
$|A|\le \|yv\|_{L^2}\|\nabla 
  v\|_{L^2}\lesssim \tau(t)$.\\

Use the conservation of the energy of $u$, and rewrite the energy in
terms of $v$, via \eqref{eq:uvDMJ}:
\begin{align*}
  \frac{d}{dt}\Bigg(&
\underbrace{E_{\rm kin}}_{=\O(1)}+\frac{(\dot
                      \tau)^2}{2}\int |y|^2|v|^2-\underbrace{\frac{\dot 
\tau}{\tau}A}_{=\O(\dot \tau)}+\underbrace{\lambda \int |v|^2\ln|v|^2}_{=\O(1)}-\lambda d\ln \tau \int
                      |v|^2\\ 
&+
\underbrace{2\lambda
                                 \|\gamma\|_{L^2}^2\ln\(\frac{\|u_0\|_{L^2}}{\|\gamma\|_{L^2}}\)}_{=\O(1)}
\Bigg)=0 ,
\end{align*}
where we have used \eqref{eq:apv} for the a priori estimates. We infer
\begin{equation*}
 \frac{(\dot \tau)^2}{2} \int |y|^2|v|^2
   -\lambda d\ln \tau \int |v|^2=\O(\dot \tau).
\end{equation*}
But multiplying \eqref{eq:tau-libre} by $\dot \tau$ and integrating, we
find $(\dot \tau)^2 = 4\lambda \ln \tau$, and
$\|v\|_{L^2}^2=\|\gamma\|_{L^2}^2 =
\frac{2}{d}\|y\gamma\|_{L^2}^2$, so we conclude
\begin{equation*}
 \|yv(t)\|_{L^2}^2-\|y\gamma\|_{L^2}^2=\O\(\frac{1}{\sqrt{\ln
       t}}\) .
 \end{equation*}
 \subsubsection{Universal profile}\label{sec:profil}
 The proof of the weak convergence of $|v(t,\cdot)|^2$ to $\gamma^2$
 relies on a hydrodynamical formulation of \eqref{eq:v}, based on
 the 
Madelung transform, which relates (nonlinear) Schr\"odinger equations
to some equations from compressible fluid mechanics (see for instance the
survey \cite{CaDaSa12}). Formally, this
amounts to a polar factorization of 
$v$,
\[v=\sqrt\rho e^{i\phi}.\]
The fluid velocity is then given by $\nabla \phi$. However, such a
decomposition is obviously delicate when $v$ (or, equivalently,
$\rho$) becomes zero. The rigorous approach consists in introducing
\[\rho=|v|^2,\quad J = \IM \bar v\nabla v.\]
 From a fluid mechanical perspective, we consider the momentum $J$
 instead of the velocity: this is standard in compressible fluid
 mechanics. Plugging this decomposition into \eqref{eq:v} and
 separating real and imaginary parts, we find:
 \begin{equation*}
 \left\{
\begin{aligned}
&\d_t \rho + \frac{1}{\tau^2}\nabla\cdot J=0,\\
& \d_t J +\lambda \nabla \rho +2\lambda 
y \rho =\frac{1}{4\tau^2}\Delta \nabla\rho
-\frac{1}{\tau^2}\nabla\cdot \RE \(\nabla
  v\otimes\nabla \bar v\).
\end{aligned}
\right.
\end{equation*}
To guess the result, consider the
baby model:
  \begin{equation*}
  \left\{
\begin{aligned}
&\d_t \rho + \frac{1}{\tau^2}\nabla\cdot J=0, \\
& \d_t J +\lambda \nabla \rho +2\lambda
y \rho =0 .
\end{aligned}
\right.
\end{equation*}
We can write an equation involving $\rho$ only, by just writing $\d_t
\nabla\cdot J=\nabla\cdot \d_t J$: 
\[\d_t \(\tau^2\d_t \rho\) =\lambda
  \nabla\cdot \(\nabla+2y\)\rho=: \lambda L\rho ,\]
where $L$ is the Fokker-Planck operator associated to the harmonic potential.
Note that $\tau^2\ll (\dot \tau \tau)^2$: define~$s$ such that
$\dis
\frac{\dot \tau \tau}{\lambda}  \d_t=\d_s \,, 
$
that is
\begin{equation*}
  s = \int\frac{\lambda}{\dot \tau \tau}= \int \frac{\ddot \tau}{2
    \dot \tau} = \frac{1}{2}\ln \dot \tau(t) \, .
\end{equation*}
Notice that
\begin{equation*}
  s\sim \frac{1}{4} \ln \ln t \, ,\quad t\to \infty \, .
\end{equation*}
Then again discarding formally lower order terms we find 
\begin{equation*}
\d_s \rho = L\rho .
\end{equation*}
\begin{remark}
  Recall that $\rho(t,y) = |v(t,y)|^2$: the previous computations have shown
\begin{equation*}
   \int_{\R^d}
  \begin{pmatrix}
    1\\
y\\
|y|^2
  \end{pmatrix}
\rho(t,y)dy=
 \int_{\R^d}
  \begin{pmatrix}
    1\\
y\\
|y|^2
  \end{pmatrix}
\gamma^2(y)dy +\O\(\frac{1}{\sqrt{\ln t}}\).
\end{equation*}
\end{remark}
We have just derived formally:
\begin{equation*}
 \d_s \rho = L\rho ,\quad L=\nabla\cdot
   \(\nabla+2y\). 
\end{equation*}
For such Fokker--Planck equation, convergence to equilibrium is known
thanks to \eqref{eq:apv}  (\cite{AMTU}),
\begin{equation*}
  \|\rho(s)-\gamma^2\|_{L^1}\lesssim e^{-Cs}\|\rho_0-\gamma^2\|_{L^1}. 
\end{equation*}
The constant $C$ stems from a spectral gap, which is, in the present
case of a Fokker-Planck operator associated to the harmonic potential, $C=2$. 
Both aspects coincide, since
\begin{equation*}
  s\sim \frac{1}{4} \ln \ln t \, ,\quad t\to \infty,\quad
  \text{hence}\quad e^{-2s}\sim \frac{1}{\sqrt{\ln t}}.
\end{equation*}
This is a hint that the new time variable $s$ is well adapted. 
 Back to the complete hydrodynamical system,  introduce the time
 variable $s$, $\tilde 
\rho(s,y):= \rho(t,y)$:
 \begin{equation*}
  \d_s \tilde \rho -\frac{2\lambda}{(\dot \tau)^2}\d_s \tilde \rho
  +\frac{\lambda }{(\dot \tau)^2} \d_s^2 \tilde \rho = L\tilde \rho
  -\frac{1}{4\lambda \tau^2}\Delta^2 \tilde \rho 
  -\frac{1}{\tau^2}\nabla\cdot \nabla\cdot \RE \(\nabla
  v\otimes\nabla \bar v\). 
  \end{equation*}
For $s\in [-1,2]$ and an arbitrary sequence $s_n\to \infty$, set 
$\tilde\rho_n(s,y) = \tilde \rho(s+s_n,y).$ 
By De la Vall\'ee-Poussin and 
Dunford--Pettis theorems, we have some weak compactness in $L^1$,
hence, up to  a subsequence,
\[  \tilde \rho_n\rightharpoonup \tilde \rho_\infty
  \text{ in }L^p_s(-1,2;L^1_y),\quad \forall p\in [1,\infty).\]
Passing to the limit in the equation for $\tilde \rho$ (see
\cite{CaGa18} for details),
\begin{equation*}
 \d_s \tilde \rho_\infty  =L\tilde \rho_\infty\text{ in
        }\mathcal D'\((-1,2)\times\R^d\).
\end{equation*}
Since $J=\IM \bar v\nabla_yv$,  \eqref{eq:integralkin} yields
\begin{equation*}
  \frac{\dot \tau}{\tau}\tilde J\in L^2_sL^1_y,\quad \text{hence }
   \frac{\dot \tau}{\tau}\nabla \cdot\tilde J_n \Tend n\infty 0\quad\text{in } L^2(-1,2;W^{-1,1}).
 \end{equation*}
Therefore, $ \d_s \tilde \rho_{\infty}=0$.

On the other hand, as evoked above, it is known from \cite{AMTU} that any solution to 
\[
\d_s \tilde \rho_\infty  =L\tilde \rho_\infty
\]
satisfying the  a priori estimates of Lemma~\ref{lem:APv} converges for
 large time
\begin{equation*}
\lim_{s\to \infty} \| \tilde \rho_\infty (s) - \gamma^2\|_{L^1({\mathbb R}^d)} = 0.
 \end{equation*}
Since we have seen that $\d_s \tilde
\rho_\infty=0$, we infer 
$\tilde \rho_\infty= \gamma^2$.
Thus, the limit is unique, and no extraction is needed:
\begin{equation*}
 \tilde\rho(s)
    \mathop{\rightharpoonup}\limits_{s\to \infty} \gamma^2 \quad 
  \text{weakly in }L^1(\R^d).
\end{equation*}
\begin{remark}
  Some information is lost when approximating the original
  hydrodynamical system by a Fokker-Planck equation: this is the
  reason why only a weak convergence is obtained. This should not be
  too surprising, as the Fokker-Planck equation is parabolic, while we
  started from a Hamiltonian equation. On the other hand, in
  \cite{FeAPDE}, by changing the strategy of proof,
  the convergence is improved: Denoting by  $  W_1$ the Wasserstein distance,
there exists $C$ such that
\begin{equation*}
 W_1\(\frac{|v(t)|^2}{\pi^{d/2}},\frac{\gamma^2}{\pi^{d/2}}\)\le
  \frac{C}{\sqrt{\ln t}},\quad t\ge e.
\end{equation*}
We recall that for $\nu_1$
and $\nu_2$ probability measures, the Wasserstein distance is defined by
\begin{equation*}
  W_p(\nu_1,\nu_2)=\inf \left\{ \(\int_{{\mathbb R}^d\times
    {\mathbb R}^d}|x-y|^pd\mu(x,y)\)^{1/p};\quad (\pi_j)_\sharp \mu=\nu_j\right\},
\end{equation*}
where $\mu$ varies among all probability measures on ${\mathbb R}^d\times
{\mathbb R}^d$, and $\pi_j:{\mathbb R}^d\times {\mathbb R}^d\to {\mathbb R}^d$ denotes the canonical
projection onto the $j$-th factor. See e.g. \cite{Vi03}.
In the case $p=1$, the Wasserstein distance, corresponding to the
Kantorovich-Rubinstein metric, is also characterized by 
\begin{equation*}
  W_1(\nu_1,\nu_2)=\sup \left\{ \int_{{\mathbb R}^d}\Phi
    d(\mu_1-\mu_2),\ \Phi\in C(\R^d;\R),\ \operatorname{Lip}(\Phi)\le 1\right\},
\end{equation*}
and it is this point of view which is adopted in \cite{FeAPDE}. The
question of the strong convergence, in L$^1$,  of $|v|^2$ toward $\gamma^2$ 
remains open for non-Gaussian initial data. 
\end{remark}

\section{From NLS to compressible fluids}\label{sec:transition}

We have seen that the end of the proof of
Theorem~\ref{theo:logNLSdisp} relies on a hydrodynamical point of
view. This suggests that we might consider models from fluid mechanics
from the very start (instead of \eqref{eq:logNLS}), and see how what
has been understood on the Schr\"odinger side can be exported to the
fluid mechanical side. Essentially, Theorem~\ref{theo:logNLSdisp} has
an exact counterpart in fluid mechanics, up to two important remarks:
\begin{itemize}
\item The direct analogue of \eqref{eq:logNLS} in fluid mechanics is
  the Korteweg equation (via Madelung transform): we may have or not
  have the capillary term (Korteweg or Euler), and we may add a
  quantum Navier-Stokes term.
\item The existence theory is much easier in the Schr\"odinger case,
  \eqref{eq:logNLS}, than for fluids. 
\end{itemize}

Consider the solution $u$ to \eqref{eq:NLS}, and resume the Madelung
transform, now directly for $u$:
\[\rho=|u|^2,\quad j = \IM \bar u\nabla u.\]The unknown $(\rho,j)$ solves
 the Korteweg system:
 \begin{equation*}
   \left\{
     \begin{aligned}
       &\d_t \rho +\nabla\cdot j=0,\\
       &\d_t j + \nabla\(\frac{j\otimes j}{\rho}\) +\nabla
       \(\rho^\gamma\)= \frac{1}{2}\rho\nabla \(\frac{\Delta
         \sqrt\rho}{\sqrt\rho}\),
     \end{aligned}
     \right.
 \end{equation*}
provided that we require
 \begin{equation*}
    \lambda = \frac{\gamma}{\gamma-1},\quad \si =\frac{\gamma-1}{2} .
  \end{equation*}
The capillarity term (right-hand side of the second equation), involving the term
$ \frac{1}{2} \(\frac{\Delta
         \sqrt\rho}{\sqrt\rho}\)$, also known as
quantum pressure or Bohm potential in quantum
mechanics, can be written in several fashions, e.g.:
\begin{align*}
  \rho\nabla \(\frac{\Delta
         \sqrt\rho}{\sqrt\rho}\)&=\frac{1}{2}\nabla\cdot\(\rho
       \nabla^2\ln \rho\)=\nabla\cdot\( \sqrt\rho
                                  \nabla^2\sqrt\rho -\nabla\sqrt\rho\otimes\nabla\sqrt\rho\) \\
  &=
       \frac{1}{2}\nabla\Delta\rho -2\nabla\cdot\(
       \nabla\sqrt\rho\otimes\nabla\sqrt\rho\) . 
     \end{align*}
 See for instance \cite{AnMa09,CaDaSa12}. Either of these formulas may
 be used, typically when constructing solutions to the Korteweg
 equation, according to the level of regularity considered, and the
 presence or absence of vacuum.
 \smallbreak

 When $\gamma>1$, the pressure law $P(\rho)=\rho^\gamma$ corresponds
 to polytropic fluids, while for $\gamma=1$, the fluid is
 isothermal. We note that to get a correspondence with fluid
 mechanics, the nonlinearity in Schr\"odinger equations comes with some
 coupling constant $\lambda>0$ (defocusing case).

  \section{The limit $\gamma\to 1$}\label{sec:limit}

  From the above identification between $\si$ and $\gamma$, passing to
  the limit $\gamma\to 1$ is clear, at least formally, in the
  equations from fluid mechanics. On the other hand, it is not obvious
  to determine the ``natural'' limit for Equation~\eqref{eq:NLS} when
  $\si\to 0$. Madelung transform, as we have seen before, suggests
  that the ``good'' limit is
  \[|u|^{2\si}u\to \ln(|u|)u\quad\text{as}\quad 
  \si\to 0.\]

  We mention \cite{WangZhang2019}, where it is shown that the ground state of
  \begin{equation*}
    -\frac{1}{2}\Delta \phi+\omega\phi = |\phi|^{2\si}\phi
  \end{equation*}
  converges, as $\si\to 0$, to the ground state of
  \begin{equation*}
    -\frac{1}{2}\Delta \phi+\omega\phi = \phi\ln|\phi|,
  \end{equation*}
  that is, the Gausson (up to invariants). This case, corresponding to
  the assumption $\lambda<0$, gives more credit to the above
  discussion. 

  Apart from this very specific case, it is difficult to give a
  rigorous meaning to the limit $\gamma\to 1$, or even construct
  solutions the case $\gamma=1$. In the case of \eqref{eq:NLS}, we
  have seen that the (nonlinear) potential energy is
  \begin{equation*}
    \frac{\lambda}{\si+1}\int_{\R^d}|u(t,x)|^{2\si+2}dx,
  \end{equation*}
  and becomes, in the case of \eqref{eq:logNLS},
  \begin{equation*}
    \lambda \int_{\R^d}|u(t,x)|^2\(\ln|u(t,x)|^2 -1\)dx.
  \end{equation*}
It has no longer a definite sign. In the fluid case, using the
conservation of mass, the standard entropy in the isothermal case
reads
\begin{equation*}
  \int_{\R^d}\rho(t,x)\ln \rho(t,x)dx,
\end{equation*}
and we naturally face the same issue. There is however a major
difference regarding the Cauchy problem: \eqref{eq:NLS} is semilinear
(for $\si<\frac{2}{(d-2)_+}$, it is solved in $H^1(\R^d)$ by using a
  fixed point argument, and the nonlinearity is viewed as a
  perturbation, see e.g. \cite{CazCourant}), while the above Korteweg
  equation is quasilinear (nonlinear terms cannot be viewed as
  perturbations, unless one works with analytic regularity). The Cauchy problem is in general still a
  major issue for the equations of compressible fluid mechanics which
  we now discuss, in the sense that the optimal assumptions to
  construct weak solutions are not always known; see e.g. \cite{RoussetBBK} and
  references therein. For this reason, we distinguish rigidity results
  (``if theorem'') and the construction of weak solutions.
  \bigbreak

  On the other hand, the presence of a pressure term of isothermal
  form in the large time limit can be guessed as follows. Consider
  more generally a barotropic  (convex) pressure law $P(\rho)$,
  not necessarily equal to $\rho^\gamma$. Since the gradient of the
  pressure is involved, the value of $P(0)$ is irrelevant from a
  mathematical point of view, and we assume $P(0)=0$. If the density
  $\rho$ is dispersive in the large time limit, then the Taylor
  expansion of $P$ at zero determines the large time behavior:
  \begin{equation*}
    P(\rho)\Eq \rho 0 P'(0)\rho +\frac{1}{2}P''(0)\rho^2 +\dots
  \end{equation*}
  If $P'(0)>0$, then  isothermal effects are present at leading
  order, while if $P'(0)=0$, the dynamics corresponds to polytropic
  fluids. This is another way, probably more natural, to interpret
  Remark~\ref{rem:perturb}; see Remark~\ref{rem:pressure}. 
 
\section{Isothermal fluids: setting}
From now on, we no longer write any Schr\"odinger equation, and $u$
denotes the fluid velocity, whose rigorous definition requires some
care (as we have slightly evoked before), and which corresponds to the momentum divided by the density,
\begin{equation*}
  u =\frac{j}{\rho},
\end{equation*}
outside of vacuum, that is for $\rho>0$ ($\rho\ge 0$ in general). We
consider 
\begin{equation}\label{eq:fluide}
  \left\{
\begin{aligned}
& \d_t \rho + \nabla\cdot (\rho u) = 0, \\
& \d_t(\rho u) + \nabla\cdot (\rho u \otimes u) + \nabla \rho
= \frac{\eps^2}{2} \rho \nabla \left( \frac{\Delta \sqrt{\rho}}{\sqrt{\rho}}\right) 
+ \nu \nabla\cdot  (\rho {\mathbb D} u) ,   
\end{aligned}
\right.
\end{equation}
with a capillarity $\eps\ge 0$, a viscosity $\nu\ge 0$, and where
${\mathbb D} u=\frac{1}{2}(\nabla u+\nabla u^{\top})$ denotes the
symmetric part of $\nabla u$. The first term of the right-hand side corresponds to
capillarity (Korteweg term), and the second is a quantum Navier-Stokes
correction, see \cite{BruMe10}: contrary to the Newtonian case involving
$\nu\Delta u$ (see e.g. \cite{Fei04,Lio98}), the viscosity can be thought of as
linear in $\rho$; see \cite{BrJa18,BrVaYu-p} for more general models
and their analysis.

\smallbreak

We shall not detail here the notion of solution adopted in
\cite{CCH18,CCH-AIF}, and present the main results or ideas in a rather
superficial way.
\smallbreak

Formally, the mass is conserved in \eqref{eq:fluide},
\begin{equation*}
  \frac{d}{dt}\int_{\R^d}\rho(t,x)dx=0,
\end{equation*}
and the energy
 \begin{equation}\label{eq:energy-init}
    E(t) = \frac{1}{2}\int_{\R^d}  \rho|u|^2 d x+\frac{\eps^2}{2} \int |\nabla
    \sqrt{\rho}|^2 dx+\int_{\R^d} \rho\ln\rho \,dx,
  \end{equation}
satisfies
\begin{equation}\label{eq:dissip-fluid}
  \dot E(t) =- \nu \int_{\R^d} \rho |\mathbb D u|^2dx.
\end{equation}
We do not write the dependence of the integrated functions upon
$(t,x)$ to shorten notations.

\begin{remark}[Explicit solutions]
  If $\rho_0$, the initial datum for
  $\rho$, is Gaussian, and if $u_0$ (initial velocity) is affine
  (think of $u_0$ as the gradient of the argument of a complex
  Gaussian), then we have explicit solutions: $\rho(t,\cdot)$ is
  Gaussian for all $t\ge 0$, $u(t,\cdot)$ is affine, and their
  time-dependent coefficients are given by explicit ordinary
  differential equations. Surprisingly enough, at leading order, the
  large time behavior 
  of the solutions to these  ordinary
  differential equations is independent of $\eps,\nu\ge 0$, and the
  analysis presented in Section~\ref{sec:gaussian} is generalized in
  \cite{CCH18}. 
\end{remark}
  \section{Rigidity in isothermal fluids}

  The end of the proof of Theorem~\ref{theo:logNLSdisp} relies on a
  hydrodynamical approach, suggesting that some results remain valid
  if we start from the isothermal Korteweg equation. The argument
  presented in Section~\ref{sec:profil} suggests that the capillary
  term has no influence in the large time behavior at leading order:
  assuming $\eps>0$ or $\eps=0$ in \eqref{eq:fluide} is not expected
  to change the large time description. More surprisingly, the
  presence of the quantum Navier-Stokes correction has no influence
  either: we may suppose $\nu=0$ or $\nu>0$.
  \smallbreak

  In view of \eqref{eq:uvDMJ} and Madelung transform, we change the
  unknown functions $(\rho,u)$ to $(R,U)$ through the relations
  \begin{equation}
  \label{eq:uvFluid}
  \rho(t,x) =
  \frac{1}{\tau(t)^d}R\(t,\frac{x}{\tau(t)}\)
\frac{\|\rho_0\|_{L^1}}{\|\Gamma\|_{L^1}},\quad 
  u(t,x) = \frac{1}{\tau(t)} U \(t,\frac{x}{\tau(t)}\) +\frac{\dot \tau(t)}{\tau(t)}x,
\end{equation}
where we denote by $y$ the
spatial variable for $R$ and $U$. The function $\tau$ is the same as
in Theorem~\ref{theo:logNLSdisp}, given by \eqref{eq:tau-libre}. The
function $\Gamma$ is defined by  $\Gamma(y)=e^{-|y|^2}$; in other
words, $\Gamma=\gamma^2$ as defined in Theorem~\ref{theo:logNLSdisp}. 
The system \eqref{eq:fluide} becomes, in terms of these new unknowns,
\begin{equation}\label{eq:RU}
  \left\{
    \begin{aligned}
  & \d_tR+\frac{1}{\tau^2}\nabla\cdot \(R U\)=0,\\
    &\d_t (R U) +\frac{1}{\tau^2}\nabla\cdot (R U \otimes U)
      +2\kappa y R 
      + \nabla R  \\
&\phantom{\d_t (R w) +\frac{1}{\tau^2}\nabla   }
=\frac{\eps^2}{2\tau^2}R\nabla\( \frac{\Delta  \sqrt{R}}{\sqrt{R}}\)   
+\frac{\nu}{\tau^2} \nabla\cdot (R \mathbb D U) +\nu \frac{\dot
  \tau}{\tau} \nabla R. 
    \end{aligned}
\right.
\end{equation}
  We define the pseudo-energy $\mathcal E$ of the system \eqref{eq:RU} by
\begin{equation}\label{eq:pseudo-energy-fluide}
    \mathcal E(t) := \frac{1}{2\tau^2}\int
  R|U|^2+\frac{\eps^2}{2\tau^2} \int|\nabla \sqrt R|^2 +\int (R |y|^2
   + R \ln R) ,
  \end{equation}
which formally satisfies 
\begin{equation}\label{eq:evol-pseudo}
   \dot{\mathcal E}(t) =- \mathcal D(t) - \nu \frac{\dot
     \tau(t)}{\tau(t)^3} \int R(t,y)  \nabla\cdot U(t,y)dy ,
\end{equation}
where the dissipation $\mathcal D(t) $ is defined by
\begin{equation}\label{eq:dissip-pseudo}
\begin{aligned}
\mathcal D(t) 
 := \frac{\dot \tau}{\tau^3}\int R |U|^2 +\eps^2 \frac{\dot \tau}{\tau^3} \int |\nabla \sqrt R|^2 
+  \frac{\nu}{\tau^4} \int R|\mathbb D U|^2.
\end{aligned} 
\end{equation}
Mimicking the proof of Lemma~\ref{lem:APv}, it is natural to expect
that each term in $\mathcal E$ is bounded (recall that $\mathcal E$ is
not signed, because of the logarithm), and that $\dot{\mathcal E}$ is
integrable. This is formally a natural assumption, but as the Cauchy
problem is a delicate issue, the following result remains an ``if
theorem'' in most cases.
\begin{theorem}[\cite{CCH18}]\label{theo:temps-long}
  Let $\eps,\nu\ge 0$, and let $(R,U)$ be a global weak solution of
  \eqref{eq:RU}.
   \begin{enumerate}
  
  \item  If $\int_0^\infty \mathcal D(t) \, dt < \infty,$
    then 
\[
\int_{\R^d} yR(t,y)d y\Tend t \infty 0\quad \text{and}\quad \left|\int_{\R^d}
  (RU)(t,y)dy\right|\Tend t \infty \infty, 
\]
 unless $\int y R(0,y)dy= \int (RU)(0,y)dy =0$, a case where 
\[
\int_{\R^d} yR(t,y)dy=\int_{\R^d}
  (RU)(t,y)d y\equiv 0.
\]

\smallskip
  \item If $\displaystyle
\sup_{t \ge 0} \mathcal E(t) + \int_0^\infty \mathcal D(t)dt < \infty,$
then
$R(t, \cdot) \rightharpoonup \Gamma$ weakly in $L^1(\R^d)$ as $ t \to \infty$.

\smallskip
\item If $\displaystyle
\sup_{t \ge 0} \mathcal E(t) < \infty$ and  the energy $E$
  defined by \eqref{eq:energy-init} 
satisfies $E(t)=o\(\ln t\)$ as $t\to \infty$, then
\begin{equation*}
  \int_{\R^d}|y|^2R(t,y) dy \Tend t \infty \int_{\R^d}|y|^2 \Gamma(y) dy.
\end{equation*}
  \end{enumerate}
\end{theorem}
Essentially, the proof is based on arguments similar to those sketched
in Section~\ref{sec:elements}. As evoked above, it is a bit of a
surprise that the Navier-Stokes term goes through the arguments, and
we refer to \cite{CCH18} for details.
\begin{remark}\label{rem:pressure}
  In the same spirit as the discussion at the end of
  Section~\ref{sec:limit}, the pressure law considered in \cite{CCH18}  is
  more general than exactly isothermal: we assume that $P\in
  C^1([0,\infty[;\R_+)\cap C^2(]0,\infty[;\R_+)$, and $P$ is convex, with $P'(0)>0$. 
\end{remark}
\section{On the existence of weak solutions}
As already evoked, constructing solutions in compressible fluid
mechanics is a difficult question. In the polytropic Euler equation
($\gamma>1$), a suitable change of unknown function (consider
$\rho^{\frac{\gamma-1}{2}}$ instead of $\rho$) makes the system
  hyperbolic symmetric, so the Cauchy problem can be solved in Sobolev
  spaces with sufficiently high regularity, but finite time blow-up
  occurs, typically when starting from smooth, compactly supported
  data, \cite{MUK86,JYC90}. Global, smooth solutions are constructed
  for suitable affine velocities: \cite{Serre97,Gra98}. In the case of
  Korteweg equation, the link with nonlinear Schr\"odinger equations
  has been exploited in \cite{BDD07}, leading to further developments,
  e.g. \cite{AnMa09,AnMa12,Aud12,AH17}. In the presence of the quantum
  Navier-Stokes 
  correction, many results are available, regarding the existence of
  weak solutions, still for $\gamma>1$; see e.g. \cite{BD07,GeLeFl16,Gis-VV15,Jungel,LacroixVasseur,VasseurYuInventiones}, and
  \cite{RoussetBBK} for a survey. However, in the isothermal case
  $\gamma=1$, far less is known: we refer to
  \cite{LeFlochShelukhin2005} for the one-dimensional Euler equation,
  \cite{Jungel} for the quantum Navier-Stokes on $\T^d$ for $d\le 2$. 
\smallbreak

In \cite{CCH-AIF}, we construct weak solutions to \eqref{eq:fluide} in
the presence of viscosity, $\nu>0$. We emphasize two aspects in this
construction, which seem to be the most important contributions of
this work:
\begin{itemize}
\item We consider solutions on the whole space $\R^d$, while most of
  the previous references assume a periodic setting, $x\in \T^d$
  ($x\in \R$ in \cite{LeFlochShelukhin2005}).
\item We gain positivity properties by working on the intermediary
  system \eqref{eq:RU}. 
\end{itemize}
Both points are intimately connected, as the change of unknown
functions \eqref{eq:uvFluid} involves a time-dependent rescaling. The
reasons why most of the references consider the periodic setting $x\in
\T^d$ seem to be mostly that compactness in space then comes from
free, and integrations by parts can be performed freely. The periodic
case is also rather 
convenient for approximating, among others in Lebesgue spaces, the
initial density by a density bounded away from zero, a step which
would require some modification on $\R^d$. Note also that this property is
classically propagated  by the flow in a suitable regularized continuity equation
(see e.g. \cite{Fei04,Jungel}), and such a property is needed in the
presence of cold pressure 
and regularizing terms (see e.g. \cite{Gis-VV15,VasseurYu}). 

For these reasons, to construct a solution $(R,U)$ to \eqref{eq:RU} on
$\R^d$, we first replace $\R^d$ with a periodic box $\T^d_\ell$ of
size $\ell>0$, where $\ell$ is aimed at going to infinity at the last
step of the proof. 

We refer to \cite{CCH-AIF} for the details, and conclude this section by
pointing out another important tool, which has proven very useful in
the context of compressible Navier-Stokes equations with a density-dependent
velocity, known as BD-entropy, after \cite{BD04,Bre-De-CKL-03}. It
involves an effective velocity, which reads $U+\nu \nabla \ln R$ in
the case of \eqref{eq:RU}:
\begin{equation*}
  \EBD (R,U)
= \frac{1}{2\tau^2} \int_{\R^d} 
\left( R|U + \nu \nabla \log R|^2 + \eps^2 |\nabla \sqrt{R}|^2  \right)
+ \int_{\R^d} \left( R|y|^2 + R\log R \right).
\end{equation*}
The evolution of this BD-entropy is given formally, for $t\ge 0$, by
\begin{equation}\label{eq:EBD}
\begin{aligned}
  \EBD(R , U)(t)  &+\int_0^t\DBD(R , U)(s) d s \\
   &=\EBD(R_0 , U_0) + \nu \int_0^t \frac{2 d}{\tau^2} \int_{\R^d} R
  + \nu \int_0^t \frac{\dot \tau}{\tau^3} \int_{\R^d} R \nabla\cdot U
  , 
  \end{aligned}
\end{equation}
where the above dissipation is defined by
\begin{equation}\label{eq:DBD}
\begin{aligned}
  \DBD(R,U)
&= \frac{\dot \tau}{\tau^3} \int \left( R|U|^2 + \eps^2 |\nabla
  \sqrt{R}|^2 \right) 
+ \frac{\nu}{\tau^4} \int_{\R^d} R |\mathbb A U|^2  \\
&\quad
+\frac{\nu \eps^2}{\tau^4}   \int R|\nabla^2  \log R|^2
+\frac{4 \nu}{\tau^2} \int |\nabla \sqrt{R}|^2 ,
\end{aligned}
\end{equation}
with $\mathbb AU := \frac12 (\nabla U - \nabla U^\top)$ the skew-symmetric part of $\nabla U.$
Hence putting together the energy and the BD-entropy equalities, it holds
\begin{equation*}
  \EE(t)  + \EBD(t) +\int_0^t \(\DD(s) + \DBD(s) \) ds = \EE(0) + \EBD(0) 
  + \nu \int_0^t \frac{2 d}{\tau^2} \int_{\R^d} R , \quad t\ge 0.
\end{equation*}
Thanks to the conservation of mass and the fact that
$\int_0^\infty \tau^{-2}(t) dt < \infty$, the last term is
uniformly bounded.   
\begin{theorem}[\cite{CCH-AIF}]
   Assume $\nu>0$, $\eps\ge 0$.
   Let $(\sqrt{R_0} , \Lambda_0 = (\sqrt{R} U)_0 ) \in L^2(\R^d) \times L^2(\R^d)$ satisfy $\EE(0) < \infty$, $\EBD(0) < \infty$,
as well as the compatibility conditions
\[
\sqrt{R_0} \ge 0 \text{ a.e.\ on } \R^d,  \quad   (\sqrt{R}U)_0=
0 \text{ a.e.\ on } \{\sqrt{R_0} = 0 \}.
\]
There exists  at least one global weak solution to \eqref{eq:RU}, which
satisfies moreover the energy and BD-entropy inequalities:
There exist absolute constants $C,C'$ such that, for almost all $t
\ge 0$,  there holds: 
\begin{align}  \label{eq:EE+DD<=EE0}
\mathcal E(t) + \int_0^t \mathcal D(s) ds &  \le C (\mathcal E(0)), \\
\label{eq_EBD_weak}
\mathcal E_{\mathrm{BD}}(t) + \int_0^t \mathcal D_{\mathrm{BD}}(s) d s &\le C' (\mathcal E(0), \mathcal E_{\mathrm{BD}}(0)) ,
\end{align}
with $\mathcal E,\mathcal D,\mathcal E_{\mathrm{BD}},\mathcal
D_{\mathrm{BD}}$ as defined in
\eqref{eq:pseudo-energy-fluide}-\eqref{eq:evol-pseudo}-\eqref{eq:EBD}-\eqref{eq:DBD}.  
\end{theorem}
This result implies existence results for \eqref{eq:fluide}, see
\cite{CCH-AIF}. Note however that this approach does not seem to
provide any relevant information regarding the energy
\eqref{eq:energy-init}. 
\section{From isothermal to polytropic}

The method of proof developed to study \eqref{eq:logNLS} and
\eqref{eq:fluide} turns out to bring some information in the case of
\eqref{eq:NLS} and polytropic fluids, as shown in
\cite{CCH-p}. Replace \eqref{eq:tau-libre} with
\begin{equation}\label{eq:tau}
     \ddot \tau = \frac{\alpha}{2\tau^{1+\alpha}},\quad \tau(0)=1,\ \dot
  \tau(0)=0.
\end{equation}
We note that this ordinary differential equation was already
considered in \cite{CidDolbeault-p}, in a context very similar to the
Schr\"odinger equation considered in \cite{CCH-p}, for a different study.
The large time behavior of $\tau$ turns out to be independent of $\alpha>0$:
\begin{lemma}\label{lem:tau}
  Let $\alpha>0$. 
  The ordinary differential equation \eqref{eq:tau}
 has a unique, global, smooth solution $\tau\in C^\infty(\R;\R_+)$. In
addition, its large time behavior is given by
\begin{equation*}
  \dot \tau(t)\Tend t {\infty} 1,\quad\text{hence }\tau(t)\Eq t
  {\infty} t.
\end{equation*}
\end{lemma}
We see that the value of the parameter $\alpha>0$ does not influence
the large time behavior, at leading order. And in view of
Theorem~\ref{theo:logNLSdisp}, the behavior changes for $\alpha=0$ (if
the numerator in \eqref{eq:tau}  is not canceled!), by
a logarithmic factor (which turns out to be the key of
e.g. Corollary~\ref{cor:growth}).  All the algebra presented so far
can then be resumed: we change unknown functions as in
\eqref{eq:uvDMJ} and \eqref{eq:uvFluid}, and obtain equations
analogous to \eqref{eq:v} and \eqref{eq:RU}. The choice of $\alpha$ is
suggested by the value of $\si$ (or, equivalently,
$\gamma$). Informally, the main result for fluid dynamics in
\cite{CCH-p} is again an ``if theorem'', as in \cite{CCH18}: every
solution to the analogue of \eqref{eq:RU}, where, among others,
$\nabla R$ is replaced by $\nabla R^\gamma$, satisfying suitable
conditions,  has an asymptotic profile,
that is, there exists $R_\infty\in  \mathbb P(\mathbb R^d)$  the set
of probability measures on $\R^d$, with two finite momenta, such that
\[
R(t,\cdot) \rightharpoonup R_{\infty} \quad \text{ in }  \mathbb
P(\mathbb R^d), \ \text{ as }t\to \infty.
\]
We have in addition $R_\infty\in L^1(\R^d)$ (at least) in the
following cases:
\begin{itemize}
\item $\eps = \nu=0$ and $1<\gamma\le 1+2/d$,
\item $\eps>0$, $\nu=0$ and $ \gamma >1 $,
\item $\eps \ge 0$, $\nu>0$ and $1<\gamma\le 1+1/d$.
\end{itemize}
The results of \cite{Serre97,Gra98} in the case of the Euler equation
($\eps=\nu=0$) and the scattering results for the nonlinear
Schr\"odinger equation (for the Korteweg equation $\eps>0=\nu$) show
that unlike what has been established in the isothermal case, the
profile $R_\infty$ is not universal.

\bibliographystyle{abbrv}
\bibliography{biblio}
\end{document}